\documentclass[a4paper,10pt]{article}
\usepackage{geometry}
\usepackage[tbtags]{amsmath}
\usepackage{amssymb}
\usepackage{amsthm}
\usepackage{fancyhdr}
\usepackage{comment}
\usepackage{latexsym}
\usepackage{mathrsfs}
\usepackage{xcolor}
\usepackage{wasysym}
\usepackage{fancyhdr}
\usepackage{float}
\usepackage{graphicx}
\usepackage[numbers,sort&compress]{natbib}
\usepackage{pict2e,color}
\usepackage{tikz}
\usepackage{float}
\usepackage{caption}
\usetikzlibrary{positioning}
\usepackage{tkz-graph}
\usetikzlibrary{arrows,shapes,chains}
\usepackage{enumerate}
\usepackage{verbatim}
\allowdisplaybreaks[4]

\newtheorem{theorem}{\bf Theorem}[section]
\newtheorem{lemma}[theorem]{Lemma}

\newtheorem{corollary}[theorem]{Corollary}
\newtheorem{conjecture}[theorem]{Conjecture}
\newtheorem{claim}[theorem]{\indent Claim}
\theoremstyle{definition}
\newtheorem{definition}[theorem]{Definition}
\newtheorem{proposition}[theorem]{Proposition}

\newtheorem{example}[theorem]{Example}
\newtheorem{remark}[theorem]{Remark}

\begin{document}

\title{Hypergraphs without Subgraphs of Given Connectivity}
\author{
Jie Ma\footnote{School of Mathematical Sciences, University of Science and Technology of China, Hefei 230026, China.} \footnote{  Yau Mathematical Sciences Center, Tsinghua University, Beijing 100084, China.}~~~~~~
Shengjie Xie\footnotemark[1]~~~~~~
Zhiheng Zheng\footnotemark[1]
}
\date{\today}
\maketitle
\begin{center}
\begin{minipage}{120mm}

\begin{abstract}
In this paper, we study the problem of determining the maximum number $F_r(n,k)$ of edges in an $n$-vertex $r$-uniform hypergraph that contains no $(k+1)$-connected subgraph. 
The graph case, initiated by Mader, is a classical problem in graph theory that remains open. We first establish a limit theorem for $F_r(n,k)$ for all $k\ge r\ge 2$. As a consequence, we prove for the first time that, in Mader's problem (i.e., $r=2$), there exists a constant $c_k>0$ such that $F_2(n,k)=c_k n+O_k(1)$, and, for every $r\ge 3$, we determine $F_r(n,k)$ up to an $O(n)$ error term, thereby identifying its leading asymptotic term.
We also address a related question of Carmesin by establishing a tight bound for $r$-uniform hypergraphs with no $(k+1)$-connected subgraph on more than $Ck$ vertices for any constant $C>2$ and sufficiently large $r$, and further obtain an asymptotically tight bound in the case $C=2$. Our proof combines the separator tree method introduced by Carmesin with several new combinatorial and optimization techniques, and we conclude with related remarks and open problems.
\end{abstract}
\end{minipage}
\end{center}

\section{Introduction}
A fundamental question in graph theory asks for the maximum possible number of edges in an $n$-vertex graph that contains no $(k+1)$-connected subgraph. In 1972, Mader~\cite{2} proved that every graph with average degree at least $4k$ contains such a subgraph (on more than $2k$ vertices).
Mader's result has inspired extensive subsequent research and has found applications in diverse areas, including bipartite minors~\cite{5}, girth problems~\cite{8}, the relationship between connectivity and chromatic number~\cite{5,9,10,11}, and the study of highly connected substructures in Tur\'an type problems~\cite{4} and graph minors~\cite{7}.
In 1979, Mader~\cite{10} proved a better bound than $4k$ for large graphs
and formulated the following famous conjecture:

\begin{conjecture}[\cite{10}, Mader]\label{conj:graph}
For any integer \(k\ge 2\) and all sufficiently large \(n\), every \(n\)-vertex graph without a \((k+1)\)-connected subgraph satisfies \(e(G) \le \binom{n}{2}-\binom{n-k}{2} + \frac{k-1}{2}(n-2k)=\frac{3}{2}\bigl(k-\frac{1}{3}\bigr)(n-k).\)
\end{conjecture}

\noindent In the same paper, Mader gave the following construction $G_{q,k}$ achieving the bound. Let $k,q$ be integers with $n=qk$. Partition the vertex set into $V_0,V_1,\dots,V_{q-1}$, each of size $k$, where $V_0$ is independent and $V_1,\dots,V_{q-1}$ induce cliques. Join every vertex of $V_0$ to every vertex of $\bigcup_{i=1}^{q-1} V_i$, and add no other edges.
Twenty years later, Yuster~\cite{13} improved Mader's bound~\cite{10} to $e(G) \le \frac{193}{120}k(n-k)$. 
The current best result, due to Bernshteyn and Kostochka~\cite{1}, states that $e(G) \le \frac{19}{12}k(n-k)$. 
Improving the original result of Mader~\cite{2}, Carmesin~\cite{3} proved that every graph with average degree at least $\frac{10}{3}k$ contains a $(k+1)$-connected subgraph on more than $2k$ vertices, and showed that the bound $\frac{10}{3}k$ is sharp under the requirement that the subgraph has more than $2k$ vertices. 
(The significance of the cut-off at $2k$ vertices will be discussed in the concluding remarks.)


In this paper, we study hypergraph analogues of Mader's result and conjecture~\cite{2,10}. 
All hypergraphs considered are $r$-uniform for some fixed $r\ge 2$. 
A natural question in this direction is the following:
\begin{quote}
Determine the maximum number $F_r(n,k)$ of edges in an $n$-vertex graph or $r$-uniform hypergraph that contains no $(k+1)$-connected subgraph.
\end{quote}
Here and throughout the paper, a hypergraph $H$ is \emph{connected} if for any $x,y\in V(H)$ there exists a sequence $x=a_0,a_1,\dots,a_s=y$ such that each pair $\{a_{i-1},a_i\}$ is contained in a common hyperedge, and it is \emph{$(k+1)$-connected} if it remains connected after the deletion of any $k$ vertices.
Hypergraph connectivity has been studied extensively, including $j$-connectivity in random hypergraphs~\cite{14}; extremal results for vertex-maximal uniform hypergraphs~\cite{15}; analytic connectivity conditions~\cite{16}; minimum cut computation~\cite{17,18}; and applications to dynamical processes~\cite{19}. 
Despite these developments, deterministic density conditions guaranteeing a $(k+1)$-connected subgraph in hypergraphs remain largely unexplored.

Carmesin explicitly raised the question of extending his result in~\cite{3} to hypergraphs (see Question~8.1 in~\cite{3}).
To formalize this, for any $c\ge 0$ and $r\ge 2$, let $\mathcal{F}_{c,r}(n,k)$ denote the family of $n$-vertex $r$-uniform hypergraphs containing no $(k+1)$-connected subgraph on more than $ck+k$ vertices.
Carmesin's question is to determine
\[
\max_{H\in \mathcal{F}_{c,r}(n,k)} e(H) \mbox{~~for $c=1$}.
\]

Our main results indicate that the above questions for $r \ge 3$ differ substantially from the graph case ($r=2$).
Suppose that $H$ is an $r$-uniform hypergraph that is not $(k+1)$-connected. Then there exists a separator of size $k$ that separates the remaining vertices into two parts.
Edges intersecting the separator and both parts do not affect the connectivity of $H$; we emphasize that such edges (called \emph{bonded edges}) do not appear in the graph case.
Moreover, the set of all edges intersecting the separator of $k$ vertices forms a hypergraph containing no $(k+1)$-connected subgraph (of any size), 
which provides a lower bound $\binom{n}{r}-\binom{n-k}{r}$ for the above question of Carmesin.
It turns out that this lower bound constitutes the main term in $\max_{H\in \mathcal{F}_{c,r}(n,k)} e(H)$ for any $c\ge 0$.
This motivates the following notion, which is the main focus of the present paper.
We define the \emph{surplus} by
\begin{equation}\label{eq:surplus}
f_{c,r}(n,k) = \max_{H\in \mathcal{F}_{c,r}(n,k)} e(H) - \binom{n}{r} + \binom{n-k}{r},
\end{equation}
Note that $F_r(n,k)=\binom{n}{r}-\binom{n-k}{r}+f_{0,r}(n,k)$.
For convenience, we further define the {\emph{relative surplus}}
\begin{equation}\label{eq:re_l(S)urplus}
g_{c,r}(n,k) = \frac{r! \cdot f_{c,r}(n,k)}{k^{r-1}(n-k)}.
\end{equation}

We now present our main results. 
Our first result shows that for any \(k\geq r\geq 2\) and \(c\geq 0\), the limit \(\lim_{n\to\infty} g_{c,r}(n,k)\) exists.
This determines the function $f_{c,r}(n,k)$ up to $O(n)$.

\begin{theorem}\label{thm:limit}
For any \(k\geq r\geq 2\) and \(c\geq 0\), the limit \(g^k_{c,r}:=\lim_{n\to\infty} g_{c,r}(n,k)\) exists. 
\end{theorem}

In this limit formulation, the asymptotic version of Conjecture~\ref{conj:graph} asserts that $g^k_{0,2} = 1 - \frac{1}{k}$, while Carmesin's result~\cite{3} establishes that $g^k_{1,2} = \frac{4}{3} - \frac{1}{k}$ for all $k \ge 2$. As an immediate corollary, 
for any $n>k\geq r=2$, there exists a constant $c_k>0$ such that
\[
F_2(n,k)=c_k n+O_k(1),
\]
and for any $n > k \ge r \ge 3$, we determine the leading term of $F_r(n,k)$ as
\begin{equation}\label{equ:F_r}
    F_r(n,k)=\binom{n}{r}-\binom{n-k}{r}+O_{r,k}(n).
\end{equation}

The next result provides a general upper bound for the relative surplus. 

\begin{theorem}\label{main} 
Let \(k\geq r\geq 3\), \(c\geq 1\) and \(n\geq (\frac{2^rc}{2^r-1}+1)k\). 
Then 
\begin{equation}\label{eq:main-1} 
g_{c,r}(n,k)\leq c^{r-1}+\frac{1}{c}\Bigl(1-\frac{1}{2^{r+1}-2}\Bigr)-\frac{k}{n-k},\quad \mbox{and consequently, } g^k_{c,r}\leq c^{r-1}+\frac{1}{c}\Bigl(1-\frac{1}{2^{r+1}-2}\Bigr). 
\end{equation} 
\end{theorem}

Our third result determines the precise value of the limit $g^k_{c,r}$ for any $c\ge 1$ (asymptotically in the case $c=1$) assuming $k\gg r$ both are sufficiently large, thereby providing an answer to Carmesin's question.

\begin{theorem}\label{0}
Let $r\geq 3$. Then we have the following:
\begin{itemize}
\item If \(c = 1\), then there exist infinitely many \(k \in \mathbb{N}\) such that
    \begin{equation}\label{eq:0-1}
     g^k_{1,r} = 2 - o_r(1).
     \end{equation}
\item If \(c \ge (2r)^{1/(r-1)}\), then for infinitely many \(k \in \mathbb{N}\) the exact value of \(g_{c,r}^k\) is determined (see \eqref{eq:gcrknm} for the precise expression). Moreover, for all sufficiently large $k\gg r$, we have the asymptotic estimate:  
\begin{equation}\label{eq:main-2}
g^k_{c,r} = c^{r-1} + \frac{1}{2c}\left(1 - \frac{1}{2^{r} - 1}\right) - O_r\left(\frac{1}{k}\right).
\end{equation} 
\end{itemize}
\end{theorem}

Note that the second item holds for any fixed $c>1$ and sufficiently large $r$ (indeed, it holds that $c\ge (2r)^{1/(r-1)}$). We emphasize that, in both items, the upper bound holds for all sufficiently large $k \gg r$. However, due to the complexity of the lower bound construction, we do not aim to obtain matching lower bounds for all such $k$; instead, we determine the exact value for infinitely many $k$ and derive asymptotically matching lower bounds for all sufficiently large $k$. 
We expect that a complete determination for all such $k$ can be achieved using our approach.

\subsection{Organization and proof outline}
We now describe the organization of the paper together with a proof outline.
The starting point is to build a framework (Section~\ref{sep-tree}) that decomposes the hypergraph as follows: repeatedly remove a set of \(k\) vertices (a separator) that disconnects it, and continue on each component until every piece has at most a prescribed size (no more than \(ck+k\)). This yields a hierarchical partition of the vertex set (see Definition~\ref{def:sep_tree} for details).

We then introduce an abstract model that captures this decomposition. In this model, the total number of edges is bounded by counting the missing edges (i.e. non-edges) at each separator and within each piece. By maximizing the number of missing edges subject to the structural constraints, we first obtain an upper bound for the case where all pieces are relatively large (see Section~\ref{normal}).
If some pieces are smaller, we delete them one by one and show that the resulting loss of edges is small enough to be controlled. This constitutes the most technical part of the proof and is carried out in Sections~\ref{tiny},~\ref{delete}, and~\ref{bound}. 
Combining these arguments yields the desired upper bounds for our main results.

Explicit constructions (Section~\ref{7}) show that the bound in Theorem~\ref{0} is tight under the given conditions, and in particular asymptotically tight for large uniformity when $c=1$.

Finally, we conclude with several remarks and open problems in Section~\ref{sec:remarks}.

\section{The separator tree}\label{sep-tree}

We now analyze the structure of \(H \in \mathcal{F}_{c,r}(n,k)\) for \(c \geq 0\). Since any \((k+1)\)-connected hypergraph has at least \(k+2\) vertices, we may assume without loss of generality that \(c \geq \frac{1}{k}\).

Suppose that \(n > ck+k\). 
Since \(H\) is not \((k+1)\)-connected, there exists a separator \(S\) of size at most \(k\), together with two subsets \(A\) and \(B\), satisfying the following:
\begin{itemize}
\item [(a)] \(A \cap B = S\), \(A \cup B = V(H)\), and \(A \setminus S, B \setminus S \neq \emptyset\).
\item [(b)] No edge of \(H\) crosses between \(A \setminus S\) and \(B \setminus S\); that is, for every edge \(e \in E(H)\) with \(e \subseteq (A \cup B) \setminus S\), we have \(e \subseteq A \setminus S\) or \(e \subseteq B \setminus S\).
\end{itemize}

We may further assume that \(|S| = k\). Indeed, if \(|S| < k\), note that \(H\) has at least \(k+2\) vertices. We can add any \(k - |S|\) vertices from \((A \cup B) \setminus S\) to \(S\) while keeping at least one vertex in each of \(A \setminus S\) and \(B \setminus S\), thereby obtaining a new separator \(S'\). One can easily verify that \((S', A \cup S', B)\) still satisfies properties (a) and (b).

Now, if \(|A| > ck+k\) or \(|B| > ck+k\), then, because \(H\) contains no \((k+1)\)-connected subgraph on more than \(ck+k\) vertices, we can again find a separator and two subsets within \(H[A]\) or \(H[B]\), respectively. Repeating this decomposition process on such subsets until each has size at most \(ck+k\) eventually yields a tree-like structure. The process of separating an \( r \)-graph (or subgraph) into two parts via a separator can be viewed in reverse: we glue the two parts along the separator to reconstruct the original hypergraph. 

\begin{definition}\label{def:sep_tree}
Suppose that \(H \in \mathcal{F}_{c,r}(n,k)\). We say a rooted tree \(T_H\) is a \emph{separator tree} of \(H\) if it satisfies the following:
\begin{itemize}
\item There are two types of nodes: \emph{subgraph nodes} and \emph{separator nodes}.
\item Each subgraph node of \(T_H\) is labeled by a vertex subset \(A\) of \(H\), and each separator node is labeled by a \(k\)-vertex subset \(S\).
\item The root of \(T_H\) is labeled by \(V(H)\) (also denoted by \(H\)).
\item For every subgraph node \(A\) with \(|A| > ck + k\), it has exactly one child, which is a separator node \(S\) such that \(S\) separates \(A\).
\item For every separator node \(S\), it has a unique parent \(A\) and two children \(A_1\) and \(A_2\), where \((S, A_1, A_2)\) satisfies properties (a) and (b) with respect to the induced subgraph \(H[A]\).
\item Every subgraph node \(A\) with \(|A| \leq ck + k\) is a leaf of \(T_H\). We denote the leaves by \emph{atoms}.
\end{itemize}
Let \(A\) and \(B\) be two nodes of \(T_H\). 
We say that \(B\) is \emph{directly below} \(A\) (or \(A\) is \emph{directly above} \(B\)) if \(B\) is a child of \(A\).
We say that \(B\) is \emph{below} \(A\) (or \(A\) is \emph{above} \(B\)) if there exists a downward path from \(A\) to \(B\).
\end{definition}

From the previous discussion, we can recursively construct a separator tree by applying the decomposition process to the subgraphs of \(H\).
We state the following without a proof.
\begin{proposition}\label{prop:exist_sep_tree}
For each \(H \in \mathcal{F}_{c,r}(n,k)\), there exists a separator tree \(T_H\).
\end{proposition}

\begin{figure}[H]
\begin{center}
	\begin{tikzpicture}[thin][node distance=1cm,on grid]
\filldraw[draw=black , fill=gray, fill opacity=0.3,  line width=1pt] (0,0.7) ellipse (0.25cm and 0.8cm); 

\filldraw[draw=black, fill=gray, fill opacity=0.3, line width=1pt] (0.625,0.7) ellipse (0.9cm and 0.3cm); 
     \filldraw[draw=black, fill=gray, fill opacity=0.3] (0,0.7) ellipse (1.55cm and 0.8cm); 
         
		\draw(0.1,-0.8)node {$H$};
		\draw(-1,0.6)node {$A$};
		\draw[color=black] (0.85,1.2)node {$B$};
		\draw[color=black] (0.85,0.25)node {$C$};
		\draw[color=blue] (0.08, 2) node (S1_new) {$S_1$};
    \draw[color=blue, ->] (S1_new) -- (0, 1.3); 
		\draw[color=blue] (2,0.7) node (S2_new) {$S_2$};
    \draw[color=blue, ->] (S2_new) -- (1.3,0.7);

	\end{tikzpicture}
	\hspace{2cm}
	\begin{tikzpicture}
\begin{scope}[local bounding box=left]
        \node[circle, minimum size=1.8mm, draw=black, fill=black, line width=0.6pt] (v1) at (0,1.2) {};
        \node[circle, minimum size=1.8mm, draw=black, fill=black, line width=0.6pt] (v2) at (-0.6,0) {};
        \node[circle, minimum size=1.8mm, draw=black, fill=black, line width=0.6pt] (v3) at (0.6,0) {};
        \node[circle, minimum size=1.8mm, draw=black, fill=black, line width=0.6pt] (v4) at (0.4,-0.9) {};
        \node[circle, minimum size=1.8mm, draw=black, fill=black, line width=0.6pt] (v5) at (1.4,-0.9) {};
        
        \node[circle, minimum size=1.8mm, draw=blue!80!black, fill=blue!80!black, line width=0.6pt] (w1) at (0,0.6) {};
        \node[circle, minimum size=1.8mm, draw=blue!80!black, fill=blue!80!black, line width=0.6pt] (w2) at (1,-0.5) {};
        
        \draw[black, line width=0.8pt] (v1) -- (w1);
        \draw[black, line width=0.6pt] (w1) -- (v2);
        \draw[black, line width=0.6pt] (w1) -- (v3);
        \draw[black, line width=0.6pt] (v3) -- (w2);
        \draw[black, line width=0.6pt] (w2) -- (v4);
        \draw[black, line width=0.6pt] (w2) -- (v5);
        
        \node[font=\small\bfseries] at (0,1.6) {$H$};
        \node[font=\small\bfseries] at (-1,0) {$A$};
        \node[font=\small\bfseries] at (1.2,0) {$B \cup C$};
        \node[font=\small\bfseries] at (0.4,-1.25) {$B$};
        \node[font=\small\bfseries] at (1.4,-1.25) {$C$};
        \node[font=\small\bfseries, color=blue!80!black] at (0.25,0.9) {$S_1$};
        \node[font=\small\bfseries, color=blue!80!black] at (1.35,-0.35) {$S_2$};
    \end{scope}
	\end{tikzpicture}
\caption{On the left, we sketch an \( r \)-uniform hypergraph \( H \) recursively separated by the separators \( S_1 \) and \( S_2 \), where \( S_1 = A \cap (B \cup C) \) and \( S_2 = B \cap C \). The corresponding separator tree is shown on the right. The three atoms are labelled \( A \), \( B \), and \( C \), each of size at most \( ck+k \).}
\label{Figseparator_tree}
\end{center}
\end{figure}
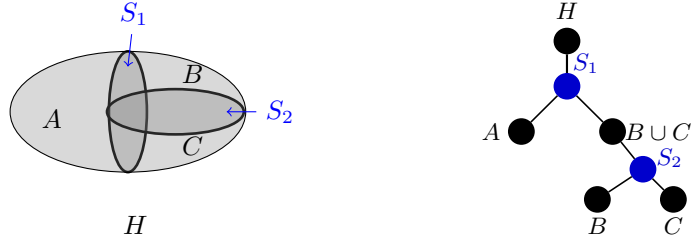

In the remaining of this section, we fix an \(H\in \mathcal{F}_{c,r}(n,k)\) and a separator tree \(T_H\) of \(H\).
Given the separator tree, we obtain an inductive relation for the number of edges in the subgraph nodes.

\begin{proposition}\label{prop:induction_edgenum}
Let \(P\) be a subgraph node that is not a leaf, let \(S\) be the separator node directly below \(P\), and let \(A\) and \(B\) be the two subgraph nodes directly below \(S\). Then
\[
    |P| = |A| + |B| - k \qquad\text{and}\qquad 
    e(P) = e(A) + e(B) + \beta(S) - e(S) ,
\] 
where \(\beta(S)\) is the number of edges that intersect all three parts \(S\), \(A \setminus S\) and \(B \setminus S\) in \(P\).    
\end{proposition}

\begin{proof}
The first equality holds because \(|S| = k\) and \(S\) separates \(A\) and \(B\).
Now consider an edge \(e\) in \(P\). Then either \(e \subseteq A\), \(e \subseteq B\), or \(e\) intersects all of \(S\), \(A \setminus S\) and \(B \setminus S\).
Since \(S = A \cap B\), the number of edges of the first two types is \(e(A) + e(B) - e(S)\).
By definition, the number of edges of the third type is \(\beta(S)\).
Summing these gives \(e(P) = e(A) + e(B) - e(S) + \beta(S)\).
\end{proof}

We introduce the notion of \emph{anti-edge} to obtain an effective upper bound for \(e(H)\). An \(r\)-subset of \(V(H)\) is called an anti-edge if it is not an edge of \(H\).

\begin{definition}\label{def:anti_edge}
For an atom \(A\), let \(\bar{e}(A) = \binom{|A|}{r} - e(A)\) denote the number of anti-edges in \(A\).
For a separator node \(S\), let \(P_1\) and \(P_2\) be the two subgraph nodes directly below \(S\), and set \(p_1 = |P_1| - k\) and \(p_2 = |P_2| - k\).
We define \(\bar{e}(S) = \binom{k}{r} - e(S)\) as the number of anti-edges in \(S\), and define
\[
\bar{\beta}(S) = \binom{p_1 + p_2 + k}{r} - \binom{p_1 + k}{r} - \binom{p_2 + k}{r} + \binom{k}{r} - \binom{p_1 + p_2}{r} + \binom{p_1}{r} + \binom{p_2}{r} 
-\beta(S)\]
By the inclusion–exclusion principle, \(\bar{\beta}(S)\) is precisely the number of anti-edges that intersect all three parts \(P_1 \setminus S\), \(P_2 \setminus S\) and \(S\) in \(P_1 \cup P_2\). We say these anti-edges \emph{bonded to} \(S\).
\end{definition}

\begin{lemma}\label{lem:edgenum_septree}
Suppose that \(\mathcal{A}\) is the set of atoms in \(T_H\) and \(\mathcal{S}\) is the set of separator nodes in \(T_H\). Then
\[
e(H) = \binom{n}{r} - \binom{n-k}{r} + \sum_{A\in \mathcal{A}}\left(\binom{|A|-k}{r} - \bar{e}(A)\right) + \sum_{S\in\mathcal{S}}\bigl( \bar{e}(S) - \bar{\beta}(S)\bigr).
\]
\end{lemma}

\begin{proof}
We prove Lemma~\ref{lem:edgenum_septree} by induction.
First, suppose that \(H\) is an atom.
Then \(\mathcal{A} = \{H\}\) and \(\mathcal{S} = \emptyset\). Substituting into the formula yields
\[
e(H) = \binom{n}{r} - \bar{e}(H) = \binom{n}{r} - \binom{n-k}{r} + \sum_{A\in \mathcal{A}}\left(\binom{|A|-k}{r} - \bar{e}(A)\right) + \sum_{S\in\mathcal{S}}\bigl( \bar{e}(S) - \bar{\beta}(S)\bigr),
\]
which verifies the base case.

Now assume that the formula holds for \(H[P]\) for every subgraph node \(P\) except \(P = H\).
Let \(S_0\) be the separator node directly below \(H\), and let \(P_1, P_2\) be the two subgraph nodes directly below \(S_0\).
Let \(\mathcal{A}_1, \mathcal{A}_2\) be the sets of atoms in \(P_1, P_2\), and let \(\mathcal{S}_1, \mathcal{S}_2\) be the sets of separator nodes in \(P_1, P_2\), respectively.
Set \(p_1 = |P_1| - k\) and \(p_2 = |P_2| - k\); then \(n = p_1 + p_2 + k\).

By Proposition~\ref{prop:induction_edgenum} and the induction hypothesis, we have
\begin{equation}\label{eq:eH_induction}
\begin{aligned}
 e(H) &= e(P_1) + e(P_2) + \bigl(\beta(S_0) - e(S_0)\bigr) \\[2mm]
&= \binom{p_1+k}{r} - \binom{p_1}{r} + \sum_{A\in \mathcal{A}_1 } \left(\binom{|A|-k}{r} - \bar{e}(A)\right) 
+ \sum_{S\in \mathcal{S}_1 } \bigl(\bar{e}(S) - \bar{\beta}(S)\bigr)  \\
&\quad\ + \binom{p_2+k}{r} - \binom{p_2}{r}+\sum_{A\in \mathcal{A}_2} \left(\binom{|A|-k}{r} - \bar{e}(A)\right) 
+ \sum_{S\in  \mathcal{S}_2} \bigl(\bar{e}(S) - \bar{\beta}(S)\bigr) 
+ \bigl(\beta(S_0) - e(S_0)\bigr).
\end{aligned}
\end{equation}

Observe that \(\mathcal{A} = \mathcal{A}_1 \sqcup \mathcal{A}_2\) and \(\mathcal{S} = \mathcal{S}_1 \sqcup \mathcal{S}_2 \sqcup \{S_0\}\). 
Since \(\bar{e}(S_0) = \binom{k}{r} - e(S_0)\), comparing the definition of \(\bar{\beta}(S_0)\) with \eqref{eq:eH_induction} yields the desired formula.
\end{proof}

From Lemma~\ref{lem:edgenum_septree}, we see that given the shape of \(T_H\) and the sizes of all atoms in \(T_H\), the number of edges \(e(H)\) depends only on
\[
W = \sum_{S\in\mathcal{S}} \bar{e}(S) - \sum_{A\in \mathcal{A}} \bar{e}(A) - \sum_{S\in\mathcal{S}} \bar{\beta}(S).
\]
We analyze the contribution of each anti-edge to \(W\).
For an anti-edge \(e\), define
\[
w(e) = \bigl|\{S\in\mathcal{S} : e \subseteq S\}\bigr| - \bigl|\{S'\in\mathcal{S} : e \text{ is bonded to } S'\}\bigr| - \bigl|\{A\in\mathcal{A} : e \subseteq A\}\bigr|.
\]
Then \(W = \sum_{\text{anti-edge } e} w(e)\).
To estimate \(w(e)\), we introduce the notion of a \emph{free anti-edge}.

\begin{definition}\label{def:orientation_and_free}
Fix an arbitrary orientation \(\sigma\) of \(T_H\); that is, for every separator node \(S \in \mathcal{S}\), we designate one of its two child subgraph nodes as the \emph{small branch} and the other as the \emph{big branch}. The distinction between the two will be specified later.
For a subgraph node \(P\) and an anti-edge \(e \subseteq P\), we define the relation between \(e\) and \(P\) recursively as follows:
\begin{enumerate}
    \item If \(P\) is an atom, then \(e\) is \emph{atomic} in \(P\).
    \item Otherwise, let \(S_0\) be the separator node directly below \(P\), let \(A\) be the small branch, and let \(B\) be the big branch.
    \begin{enumerate}
        \item If all endpoints of \(e\) belong to both \(A \setminus S_0\) and \(B \setminus S_0\), then \(e\) is \emph{free} in \(P\).
        \item If \(e\) is bonded to \(S_0\), then \(e\) is \emph{bonded} in \(P\).
        \item If \(e \subseteq B\), then the relation between \(e\) and \(P\) is the same as the relation between \(e\) and \(B\).
        \item Otherwise, \(e \subseteq A\), and the relation between \(e\) and \(P\) is the same as the relation between \(e\) and \(A\).
    \end{enumerate}
\end{enumerate}

For a separator node \(S\) and an anti-edge \(e \subseteq S\), we say that \(e\) is \emph{free in \(S\)} if \(e\) is free in the small branch of \(S\). 
Let \(f_\sigma(S)\) denote the number of free anti-edges in \(S\); when no ambiguity arises, we simply write \(f(S)\).
\end{definition}

\begin{lemma}\label{lem:edge_f(S)}
For any anti-edge \(e\), we have
\begin{equation}\label{eq:free_weight}
w(e) \leq \bigl|\{S \in \mathcal{S} : e \text{ is free in } S\}\bigr|.
\end{equation}

Consequently,
\[
e(H) \leq \binom{n}{r} - \binom{n-k}{r} + \sum_{A\in \mathcal{A}} \binom{|A|-k}{r} + \sum_{S\in\mathcal{S}} f(S).
\]
\end{lemma}

\begin{proof}
By the definition of \(w(e)\), it suffices to show that for any two distinct separators \(S_1, S_2\) such that \(e\) is not free in either \(S_1\) or \(S_2\), there exist distinct subgraph nodes \(A_1, A_2\) such that the recursive process in Definition~\ref{def:orientation_and_free} applied to \((e, S_1)\) and \((e, S_2)\) terminates at \(A_1\) and \(A_2\), respectively.
If this holds, then
\begin{equation}\label{eq:not_free}
\bigl|\{S \in \mathcal{S} : e \text{ is not free in } S\}\bigr| \leq \bigl|\{S'\in\mathcal{S} : e \text{ is bonded to } S'\}\bigr| + \bigl|\{A\in\mathcal{A} : e \subseteq A\}\bigr|.
\end{equation}
Indeed, for each \((e, S_i)\) with \(e\) not free in \(S_i\), there exists a corresponding subgraph node \(A_i\). If \(A_i\) is an atom, then \(e \subseteq A_i\); if \(A_i\) is not an atom, then \(e\) is bonded to the separator node \(S_i'\) directly below \(A_i\). In either case, the pair \((e, A_i)\) or \((e, S_i)\) is counted at most once on the right-hand side.
Clearly, \eqref{eq:not_free} implies \eqref{eq:free_weight}, and the lemma follows.

Now suppose, for contradiction, that \(A_1 = A_2\).
For \(i = 1, 2\), let \(P_i\) be the small branch of \(S_i\). Then \(A_i\) lies below \(P_i\), and consequently one of \(P_1\) and \(P_2\) is below the other. Without loss of generality, assume that \(P_1\) is below \(P_2\).

Consider the recursive process starting from \(P_2\). Since \(A_1 = A_2\), the process must pass through \(P_1\). However, because \(e \subseteq S_1\), this anti-edge lies entirely within the big branch of \(S_1\) (since \(S_1\) is a separator and \(e\) is contained in it). According to Definition~\ref{def:orientation_and_free} 2(c), when the process encounters \(S_1\), it proceeds into the big branch of \(S_1\) rather than into \(P_1\). Hence the process cannot reach \(P_1\), contradicting the assumption that \(A_1 = A_2\). Therefore, \(A_1\) and \(A_2\) must be distinct.
\end{proof}

Equipped with Lemma~\ref{lem:edge_f(S)}, to estimate the number of edges in \(H\), we only need to determine:
\begin{itemize}
\item The shape of \(T_H\) and the sizes of all atoms in \(T_H\); and
\item The quantity \(\sum_{S\in\mathcal{S}} f(S)\).
\end{itemize}
This is why we do not need to concern about the specific structure inside each subgraph node.
Instead, we aim to wisely orient \(T_H\) and derive an upper bound for \(f(S)\) for each separator.

\begin{definition}\label{def:abstract_tree}
Suppose \(T_H\) is a separator tree. 
We define \(\widehat{T_H}\) from \(T_H\) as follows:
\begin{itemize}
\item For each subgraph node \(A\), we retain only its size \(|A|\);
\item For each separator node \(S\), we no longer regard it as an actual separator within a subgraph of \(H\).
\end{itemize}
For a given abstract tree \(\widehat{T}\) and an orientation \(\sigma\), let \(\widehat{F}_{\sigma}\) denote the maximum number of free anti-edges over all hypergraphs \(H' \in \mathcal{F}_{c,r}(n,k)\) for which \(\widehat{T}\) is an abstract tree of \(H'\).
We then define the edge estimate
\[
e(\widehat{T_H}) = \binom{n}{r} - \binom{n-k}{r} + \sum_{A\in \mathcal{A}} \binom{|A|-k}{r} +  \widehat{F}_{\sigma}.
\]
When no ambiguity arises, we simply write \(\widehat{F} = \widehat{F}_{\sigma}\).
\end{definition}

The following result is an immediate consequence of Lemma~\ref{lem:edge_f(S)}.

\begin{corollary}\label{cor:edge_abstract_tree}
For any \(H\) with separator tree \(T_H\), we have \(e(H) \leq e(\widehat{T_H})\).
\end{corollary}

\section{Normal Atoms}\label{normal}
In section~\ref{normal} to section~\ref{bound}, we always assume that \(c\geq 1\).

Now we give a rule to determine which branch of a separator is considered big or small. We begin by considering the size of atoms. An atom is \emph{normal} if its size is at least $\frac{4k}{3}$; otherwise, it is called \emph{tiny}. The choice of the threshold $\frac{4k}{3}$ will be justified in the proof.

A separator tree or abstract tree is \emph{normal} if all its atoms have size greater than $\frac{4k}{3}$. In this section, we consider only normal abstract trees. For a branch of a separator, we say it is \emph{big} if it contains more normal atoms than the other branch; the other branch is called \emph{small}. If both branches contain the same number of normal atoms, we choose the big branch arbitrarily. This rule defines an orientation $\sigma$ of the separator tree. We begin by estimating the number of free anti-edges at separators in a normal separator tree to bound $f(S)$. For convenience, let \(\ell(S)\) and \(\ell(S,+)\) be the numbers of normal atoms in the small branch and the big branch of \(S\), respectively. We say a separator $S$ is \emph{balanced} if $\ell(S)=\ell(S,+)$.
\begin{lemma}\label{normal1}
Let \(S\) be a separator in a normal separator tree. The number \(f(S)\) of free anti-edges in \(S\) within the small branch of \(S\) is at most
\[
f(S) \le \binom{k}{r} 
- \bigl(\ell(S)-q\bigr) \cdot \binom{\bigl\lfloor\frac{k}{\ell(S)}\bigr\rfloor}{r}
- q \cdot \binom{\bigl\lceil\frac{k}{\ell(S)}\bigr\rceil}{r}\leq \binom{k}{r}-\ell(S)\cdot\binom{\frac{k}{\ell(S)}}{r} \footnote{We use the extended binomial coefficient \(\binom{x}{r}\) for non‑integer \(x\), defined for all real numbers \(x\) by \(\binom{x}{r} = \frac{x(x-1)\cdots(x-r+1)}{r!}\) when \(x \geq r-1\), and \(\binom{x}{r} = 0\) when \(x < r-1\). This convention ensures the continuity and convexity required in the subsequent analysis.},
\]
where \( q = k \bmod \ell(S) \).
\end{lemma}

\begin{proof}
Every vertex in \(S\) belongs to at least one normal atom. If an anti-edge is free, its vertices cannot all lie in the same normal atom. Assign each vertex of \(S\) to one of the normal atoms that contain it; thus the vertex set of \(S\) is partitioned into groups, each assigned to a single normal atom. The number of free anti-edges of \(S\) in the small branch is then bounded above by the number of anti-edges whose endpoints lie in different groups. This quantity equals the number of edges in a complete \(\ell(S)\)-partite \(r\)-graph whose part sizes are as balanced as possible. The latter is maximized when the parts differ in size by at most one, yielding the first inequality.
The second inequality follows from the convexity of \(h(x) = \binom{x}{r}\).
\end{proof}

In Lemma~\ref{normal1}, we focus on the small branch of a separator \(S\) because it contains fewer normal atoms. When the two branches \(A\) and \(B\) of \(S\) are glued along \(S\), the induced subgraphs \(A[S]\) and \(B[S]\) coincide. Consequently, \(S\) can only be realized as a complete \(\ell(S)\)-partite hypergraph, since the small branch has at most as many normal atoms as the big branch. 

Now we can obtain an upper bound on the number of edges of a normal separator tree. Example~\ref{ex1} will demonstrate that this bound is tight.

\begin{theorem}\label{normal2}
Let \(T_N\) be a normal separator tree with \(t + k\) vertices and let \(\mathcal{S}\) be the set of its separators. Then the number of edges \(e(T_N)\) satisfies
\begin{equation}\label{eq:normal_abstract_tree}
\begin{aligned}
e(T_N) \leq \binom{t+k}{r} - \binom{t}{r} +\frac{t}{ck}\binom{ck}{r}
        + \left(\frac{t}{ck} - 1\right)\binom{k}{r} 
        - \frac{t}{2ck} \sum_{i = 0}^{\infty} \binom{k / 2^{\,i}}{r}
        + f(k, r, L(T_N)) - X(T_N),    
\end{aligned}
\end{equation}
where \(L(T_N)\) is the number of normal atoms in \(\widehat{T_N}\), 
$f(k, r, L) = \frac{L}{2} \sum_{i = 0}^{\infty} \binom{k / (2^i L)}{r}$,
and \(X(T_N) = \sum_{S \in \mathcal{S}} X(S)\) is the branch-error term with 
\begin{equation}\label{eq:def_xs}
X(S) := f(k,r,\ell(S)+\ell(S,+))-f(k,r,\ell(S))-f(k,r,\ell(S,+))
      +\ell(S)\,\binom{\frac{k}{\ell(S)}}{r}.
\end{equation}
It will be shown in the proof that \(X(S) \ge 0\).
\end{theorem}

\begin{proof}
We proceed by induction on the branches \(P\) of \(T_N\), following the tree structure of \(T_N\). Thus, we view \(P\) as a subtree of \(T_N\). 
Let \(p = |P| - k\).
Let \(\mathcal{A}_P\) and \(\mathcal{S}_P\) be the sets of atoms and separators below \(P\), respectively.
We define
\[
\tilde{e}(P) := \binom{p + k}{r} - \binom{p}{r} + \sum_{A \in \mathcal{A}_P} \binom{|A| - k}{r} + \sum_{S \in \mathcal{S}_P} f(S).
\]
By Lemma~\ref{lem:edge_f(S)}, \(e(P)\leq \tilde{e}(P)\), hence it suffices to show that \(\tilde{e}(P)\) is at most the right-hand side of \eqref{eq:normal_abstract_tree}.

\bigskip

\noindent \textbf{The base case.}  For the base case, \(P\) consists of a single atom. 
Then \(L(P) = 1\) and \(\frac{k}{3} \le p  \le ck\).
We have \(X(T_N) = 0\) and 
\(f(k, r, L(P)) = f(k, r, 1) = \frac{1}{2} \sum_{i=0}^{\infty} \binom{k / 2^{i}}{r}.\)

Since \(\tilde{e}(P) = \binom{p + k}{r}\), to verify \eqref{eq:normal_abstract_tree} it suffices to show
\begin{equation}\label{eq:base_normal}
M_p := \frac{p}{ck}\binom{ck}{r}-\binom{p}{r}+\left(\frac{p}{ck}-1\right)\binom{k}{r}+\left(\frac{1}{2}-\frac{p}{2ck}\right)\sum_{i=0}^\infty \binom{k/2^i}{r}\geq 0.
\end{equation}
We prove it case by case.
\medskip

\noindent\textbf{Case 1: } \(k\leq p\leq ck\).  
Then \(\binom{p}{r} \le \left(\frac{p}{ck}\right)^{r} \binom{ck}{r}\) and \(\binom{k}{r} \le \frac{1}{c^{r}} \binom{ck}{r}\).  
Inserting these bounds into \eqref{eq:base_normal} and keeping only the first term of the sum, we obtain
\[
M_p\geq \Bigl[\frac{p}{ck}
- \Bigl(\frac{p}{ck}\Bigr)^{r} - \Bigl(\frac{1}{2}-\frac{p}{2ck}\Bigr)\frac{1}{c^{r}}\Bigr] \binom{ck}{r}.
\]
Hence it suffices to prove that
\(h_1\!\left(\frac{p}{ck}\right) = \frac{p}{ck} - \left(\frac{p}{ck}\right)^{r} - \left(\frac{1}{2} - \frac{p}{2ck}\right) \frac{1}{c^{r}} \ge 0\).
Set \(x = \frac{p}{ck}\). The derivative \(h_1'(x) = 1 - r x^{r-1} + \frac{1}{2c^{r}}\) is decreasing, hence it has at most one zero on \(x \in [1/c, 1]\).  
Checking the endpoints:
\(h_1(1/c) = \frac{1}{c} - \frac{1}{c^{r}} - \left(\frac{1}{2} - \frac{1}{2c}\right) \frac{1}{c^{r}} \ge 0\),
\(h_1(1) = 0\).
Thus \eqref{eq:base_normal} holds when \(k \le p \le ck\).

\medskip
\noindent\textbf{Case 2: } \(\frac{k}{2} \le p < k\).  
Then \(\binom{p}{r} \le \left(\frac{p}{k}\right)^{r} \binom{k}{r}\) and \(\binom{ck}{r} \ge c^{r} \binom{k}{r}\).  
Substituting these bounds into \eqref{eq:base_normal} and keeping only the first term of the sum, we obtain
\[
M_p\geq \Bigl[c^{r-1}\frac{p}{k} - \left(\frac{p}{k}\right)^r- \frac{1}{2} + \frac{p}{2ck} \Bigr]\binom{k}{r}.
\]
Hence it suffices to verify
\[
h_2(p) = c^{r-1} \frac{p}{k} - \left(\frac{p}{k}\right)^{r} - \frac{1}{2} + \frac{p}{2ck} \ge 0 \qquad \left(p \in \left[\tfrac{k}{2}, k\right]\right).
\]
The derivative \(h_2'(p) = \frac{c^{r-1}}{k} - \frac{r p^{r-1}}{k^{r}} + \frac{1}{2ck}\) is decreasing in \(p\) and therefore has at most one zero, so it suffices to check the endpoints:
\[
h_2\!\left(\frac{k}{2}\right) = \frac{c^{r-1}}{2} - \frac{1}{2^{r}} - \frac{1}{2} + \frac{1}{4c} \ge 0,
\qquad
h_2(k) = c^{r-1} - 1 - \frac{1}{2} + \frac{1}{2c} \ge 0.
\]
These inequality holds because for \(c \ge 1\) we have \(c^{r-1} + \frac{1}{2c} \ge \frac{3}{2}\).

\medskip
\noindent\textbf{Case 3: } \(\frac{k}{3} \le p < \frac{k}{2}\).  
We have \(\binom{p}{r} \le \left(\frac{2p}{k}\right)^{r} \binom{k/2}{r}\) and \(\binom{ck}{r} \ge c^{r} \binom{k}{r}\). Substituting these bounds into \eqref{eq:base_normal} and keeping the first two terms of the sum, we obtain
\begin{equation}
M_p\geq \Bigl(\frac{p\,c^{r-1}}{k}-\frac12 + \frac{p}{2ck}\Bigr)\binom{k}{r}
+ \Bigl(\frac12 - \frac{p}{2ck} - \Bigl(\frac{2p}{k}\Bigr)^{\!r}\Bigr)
  \binom{\frac{k}{2}}{r}.
\label{eq:case3_bound}
\end{equation}
Since \(\frac{p c^{r-1}}{k} - \frac{1}{2} + \frac{p}{2ck} \ge \frac{3p}{2k} - \frac{1}{2} \ge 0\) for \(p \ge k/3\) (recall \(c \ge 1\)) and $\binom{k}{r}\ge 2^r \binom{k/2}{r}$, the right-hand side of \eqref{eq:case3_bound} is bounded below by
\[
\left[ \left( -\frac{1}{2} + \frac{3p}{2k} \right) 2^{r}
      + \frac{1}{2} - \frac{p}{2k} - \left(\frac{2p}{k}\right)^{r} \right] \binom{k/2}{r}.
\]
Define
\[
h_3\left(\frac{p}{k}\right) = \left( -\frac{1}{2} + \frac{3p}{2k} \right) 2^{r}
        + \frac{1}{2} - \frac{p}{2k} - \left(\frac{2p}{k}\right)^{r}.
\]
It suffices to show \(h_3(x)\geq 0\) for \(\frac{1}{3}\leq x \leq \frac{1}{2}\) For \(r \ge 3\), the derivative satisfies
\[
h_3'(x) = 3 \cdot 2^{r-1}- \frac{1}{2} - 2^r\cdot r\cdot x^{r-1} 
        \ge   3 \cdot 2^{r-1} - \frac{1}{2} - 2r > 0,
\]
so \(h_3\) is increasing on \([1/3, 1/2)\). Moreover,
\(h_3(1/3) = \frac{1}{3} - \left(\frac{2}{3}\right)^{r} \ge 0\) for \(r \ge 3\).
Hence \(h_3(p) \ge 0\) for all \(p \in [k/3, k/2)\), which establishes \eqref{eq:case3_bound}.

Combining these three cases, inequality \eqref{eq:base_normal} holds for every \(p \in [k/3, ck]\).

\bigskip
\noindent \textbf{The induction step.} For the induction step, let \(P\) be a branch whose abstract tree contains at least one separator node. Let \(S\) be the separator node directly below \(P\), and let \(A\) and \(B\) be the two branches of \(S\) with \(A\) being the small branch. 
Let \(a=|A|-k\) and \(b=|B|-k\). 

Applying the induction hypothesis to \(A\) and \(B\), we obtain:
\begin{align*}
\tilde{e}(A) &\le \binom{k + a}{r} - \binom{a}{r} + \frac{a}{ck}\binom{ck}{r} + \left(\frac{a}{ck} - 1\right)\binom{k}{r} - \frac{a}{2ck} \sum_{i=0}^{\infty} \binom{k / 2^{i}}{r} + f(k, r, \ell(S)) - X(A), \\
\tilde{e}(B) &\le \binom{k + b}{r} - \binom{b}{r} + \frac{b}{ck}\binom{ck}{r} + \left(\frac{b}{ck} - 1\right)\binom{k}{r} - \frac{b}{2ck} \sum_{i=0}^{\infty} \binom{k / 2^{i}}{r} + f(k, r, \ell(S,+)) - X(B).
\end{align*}
By the structure of \(P\), we have \(p = a + b\) and \(L(P)=\ell (S)+ \ell(S,+)\). Moreover,
\begin{align}
\tilde{e}(P) &= \tilde{e}(A) + \tilde{e}(B) + \binom{a + b + k}{r} - \binom{a + k}{r} - \binom{b + k}{r} + \binom{k}{r} - \binom{a + b}{r} + \binom{a}{r} + \binom{b}{r} - \binom{k}{r} + f(S) \notag \\
&\le \binom{k + p}{r} - \binom{p}{r} + \frac{p}{ck}\binom{ck}{r} + \left(\frac{p}{ck} - 1\right)\binom{k}{r} - \frac{p}{2ck} \sum_{i=0}^{\infty} \binom{k / 2^{i}}{r} \notag \\
&\quad + f(k, r, \ell(S)) + f(k, r, \ell(S,+)) - \binom{k}{r} + f(S) - X(P) + X(S). \label{eq:edge_estimate}
\end{align}

From Lemma~\ref{normal1} we obtain the estimate
\begin{equation}
f(S) \le \binom{k}{r} - \ell(S) \cdot \binom{k / \ell(S)}{r}.\label{eq:f_estimate}
\end{equation}
Combining \eqref{eq:edge_estimate}, \eqref{eq:f_estimate} and the definition of \(X(S)\) in \eqref{eq:def_xs}, we obtain the bound in \eqref{eq:normal_abstract_tree}.
The only remaining part is to show \(X(S)\geq 0\).

\bigskip
\noindent \textbf{Proof of \(X(S) \ge 0\).}
It suffices to prove that for \(1 \le \ell \le \ell'\),
\begin{equation*}
f(k, r, \ell + \ell') - f(k, r, \ell) - f(k, r, \ell') + \ell \cdot \binom{k / \ell}{r} \ge 0.
\end{equation*}

Recall that \(f(k, r, L) = \frac{L}{2} \sum_{i=0}^{\infty} \binom{k / (2^{i} L)}{r}\). 
Fix \(k\) and \(r\), and define functions \(g_i\) by
\[
g_i(x) = x k \cdot \binom{1 / (2^{i} \cdot x)}{r},
\]
and set \(g(x) = \frac{1}{2} \sum_{i=0}^{\infty} g_i(x)\). 
Then \(f(k, r, k x) = g(x)\).
Note that \(g(x) = 0\) for \(x > \frac{1}{r-1}\). 
Thus it suffices to show that for any \(0 < x < y \le \frac{1}{r-1}\),
\begin{equation}\label{eq:g_geq_0}
g(x) + g(y) - x k \cdot \binom{1 / x}{r} \le g(x + y).
\end{equation}
We establish this inequality using the following two claims.

\begin{claim}\label{convex}
    For each \(i \in \mathbb{N}^+\) and \(x \in \mathbb{R}^+\), the function \(\tilde{g}_i(y)=g_i(y) - g_i(x + y)\) is decreasing on \(\mathbb{R}^+\).
    Consequently, \(\tilde{g}(y)=g(y) - g(x + y)\) is decreasing on \(\mathbb{R}^+\).
\end{claim}

\begin{claim}\label{equal}
    For each \(x \in \bigl(0, \frac{1}{r-1}\bigr]\),
    \[
    2g(x) = g(2x) + x k \cdot \binom{1 / x}{r}.
    \]
\end{claim}

\noindent Assuming these two claims, we prove \eqref{eq:g_geq_0} as follows.
Since \(g(y) - g(x + y)\) is decreasing in \(y\) by Claim~\ref{convex}, we have
\[
g(y) - g(x + y) \leq g(x) - g(2x).
\]
Applying Claim~\ref{equal} we obtain
\[
g(y)-g(x+y) \le g(x) - 2g(x) + y k \cdot \binom{1 / y}{r},
\]
implying \eqref{eq:g_geq_0}.

\begin{proof}[Proof of Claim~\ref{convex}]
    For \(y \in \bigl(0, \frac{1}{2^{i}(r-1)}\bigr]\), we can write
    \[
    g_i(y) = \frac{k}{r! \cdot 2^{i}} \left( \frac{1}{2^{i} y} - 1 \right) \left( \frac{1}{2^{i} y} - 2 \right) \cdots \left( \frac{1}{2^{i} y} - (r-1) \right).
    \]
    For each \(j \le r-1\), the function \(\frac{1}{2^{i} y} - j\) is non-negative, decreasing, and convex on \(\bigl(0, \frac{1}{2^{i}(r-1)}\bigr]\). Hence \(g_i(y)\), being a positive constant multiple of the product of such functions, is also non-negative, decreasing, and convex on the same interval.

    If \(x + y \le \frac{1}{2^{i}(r-1)}\), then the convexity of \(g_i\) implies that \(g_i(y) - g_i(x + y)\) is decreasing.
    Otherwise, when \(x + y > \frac{1}{2^{i}(r-1)}\), we have \(g_i(x + y) = 0\), and hence \(g_i(y) - g_i(x + y) = g_i(y)\), which is decreasing as well. This completes the proof.
\end{proof}

\begin{proof}[Proof of Claim~\ref{equal}]
    Define \(i_x = \bigl\lfloor -\log_2 x - \log_2 (r-1) \bigr\rfloor\).
    Then \(g(x)\) can be expressed as
    \[
    g(x) = \frac{x k}{2} \sum_{i=0}^{i_x} \binom{1 / (2^{i} x)}{r}.
    \]
    Since \(i_{2x} = i_x - 1\), we have
    \begin{align*}
        g(2x) &= \frac{2x k}{2} \sum_{i=0}^{i_{2x}} \binom{1 / (2^{i} \cdot 2x)}{r} = x k \sum_{i=0}^{i_x - 1} \binom{1 / (2^{i+1} x)}{r} = x k \sum_{i=1}^{i_x} \binom{1 / (2^{i} x)}{r} = 2g(x) - x k \binom{1 / x}{r}.
    \end{align*}
    Rearranging yields the desired identity.
\end{proof}
This establishes \(X(S) \ge 0\) and completes the proof of Theorem~\ref{normal2}.
\end{proof}

Note that \(L(T)\) and \(X(T)\) are invariant among separator trees that share the same abstract tree. 
For any normal separator tree \(T\) and its abstract tree \(\widehat{T}\), we define \(L(\widehat{T}) = L(T)\), \(X(\widehat{T}) = X(T)\) and
\begin{equation}\label{eq:normal_abstract_def}
e(\widehat{T}) = \binom{t+k}{r} - \binom{t}{r} + \frac{t}{ck}\binom{ck}{r}
        + \left(\frac{t}{ck} - 1\right)\binom{k}{r} 
        - \frac{t}{2ck} \sum_{i = 0}^{\infty} \binom{k / 2^{\,i}}{r}
        + f(k, r, L(\widehat{T})) - X(\widehat{T}).
\end{equation}
By Theorem~\ref{normal2}, this definition is well-defined and satisfies Corollary~\ref{cor:edge_abstract_tree}.

\section{Tiny atoms}\label{tiny}
We now proceed to the non-normal case.
For a separator tree, we first define a \emph{tiny vertex}: a vertex \(v\) is called a \emph{tiny vertex} if it belongs to a tiny atom \(A\) but not to the separator directly above \(A\). 
Thus, the number of tiny vertices in a tiny atom \(A\) is \(|A| - k\).
We may assume that if a separator \(S\) has two atoms \(A\) and \(B\) directly below it, then \(|A|+|B|\geq (c+2)k\), since otherwise we could merge them into a larger atom without deleting any edges.
As a consequence, no separator has both branches being tiny atoms. 

For a separator tree \(T\) and a separator \(S\) of \(T\), we designate its \emph{big branch} according to the following rules:
\begin{enumerate}
    \item If one branch of \(S\) contains strictly more normal atoms than the other, it is chosen as the big branch.
    \item If both branches contain the same number of normal atoms, the branch containing more tiny vertices is selected as the big branch.
    \item If the two branches are also equal in the number of tiny vertices, the big branch is chosen arbitrarily.
\end{enumerate}
The other branch is called the \emph{small branch}. The \emph{reach} of a separator is defined as the number of normal atoms in its small branch; it is denoted by \(\ell(S)\) as defined earlier. This definition yields an orientation \(\sigma\) of the separator tree.

Except that we cannot distinguish tiny vertices in an abstract tree, these definitions are also valid for abstract trees.
 
\begin{definition}
Let \(T\) be a separator tree. For each tiny vertex \(v\), let \(S_v\) be the first (lowest) separator in \(T\) satisfying:
\begin{itemize}
    \item the tiny atom containing \(v\) lies in the small branch of \(S_v\);
    \item \(v \in S_v\) and \(S_v\) is not the separator directly above the tiny atom of \(v\).
\end{itemize}
We then say that \(v\) is \emph{assigned} to \(S_v\).
\end{definition}

\begin{definition}[Technical data]
Let \(T\) be a separator tree. 
For each pair \((A,S)\), where \(A\) is a tiny atom in the small branch of a separator node \(S\) and \(S\) is not the separator node directly above \(A\), we introduce a function \(m(A,S)\) as the number of tiny vertices of \(A\) that are assigned to \(S\) in \(T\). It is straightforward to verify that this definition satisfies the following properties. 
\begin{enumerate}
   \item For a fixed tiny atom \(A\), \(\displaystyle\sum_{S \in \mathcal{S}} m(A,S)\) does not exceed the total number of tiny vertices in \(A\).
    \item For a fixed separator \(S\), \(\displaystyle\sum_{A \in \mathcal{T}} m(A,S) \le k\), the size of the separator.
\end{enumerate}
where we denote by \(\mathcal{S}\) and \(\mathcal{T}\) the sets of separators and tiny atoms in \(T\), respectively. We denote $m(S)=\sum_{A \in \mathcal{T}}m(A,S)$.
\end{definition}
\begin{lemma}\label{technical data}
 Let  \( S \) be a separator in a separator tree \(T\), the number $f(S)$ of free anti-edges in \(S\) within the small branch  of $S$ is at most
\[
f(S) \le \binom{k}{r} 
       - \bigl(\ell(S)-q\bigr)\cdot\binom{\bigl\lfloor\frac{k - m(S)}{\ell(S)}\bigr\rfloor}{r}
       - q\cdot\binom{\bigl\lceil\frac{k - m(S)}{\ell(S)}\bigr\rceil}{r}
       - \sum_{A} \binom{m(A, S)}{r},
\]
where 
\begin{itemize}
    \item  $\ell(S) \ne 0$ and \( q = \bigl(k - m(S)\bigr) \bmod \ell(S) \).
    \item the sum runs over all tiny atoms \( A \) in the small branch of $S$ that are not directly below \( S \).
\end{itemize}
Moreover, if \(\ell(S)=0\) then \(f(S)=0\).
\end{lemma}

\begin{proof}
We consider the separator tree \(T\).

First, the special case where no tiny vertex is assigned to \(S\) is handled by Lemma~\ref{normal1}.

Let \(u\) be a vertex in \(S\). 
We define a path \(P_u\) in \(T\) that starts in the small branch at \(S\) and proceeds downward as follows: at each separator node, if \(u\) appears in only one branch, \(P_u\) follows that branch; otherwise, \(P_u\) follows the big branch. 
The path ends at an atom, which we denote by \(A_u\). 
Then, for any tiny atom \(A\), we have \(A_u = A\) if and only if \(u\) is counted in \(m(A,S)\).

Recall that \(f(S)\) is bounded by the number of anti-edges that are free in \(S\) within the small branch of \(S\). 
By Definition~\ref{def:orientation_and_free}, an anti-edge \(e\subseteq S\) is free in \(S\) only if it first belongs to the small branch of \(S\), and then, moving downward through the separator tree, its vertices are eventually separated by some lower separator into different parts. 
Therefore, if all the vertices of \(e\) always remain in the same part below \(S\), then \(e\) is not free in \(S\). 
Consequently, every anti-edge \(e\) that is free in \(S\) must contain two vertices \(u,v\) with \(A_u\neq A_v\).

We distribute the vertices of \(S\) according to the atoms to which they belong. Among the \(k\) vertices of \(S\), exactly \(m(S)\) belong to tiny atoms, with \(m(A,S)\) vertices belonging to each tiny atom \(A\), while the remaining \(k - m(S)\) vertices belong to the \(\ell(S)\) normal atoms in the small branch of \(S\).

An anti-edge contained entirely in the class of a single atom cannot be free in \(S\). 
Thus, to obtain an upper bound on the number of free anti-edges, we subtract all \(r\)-sets lying completely inside one such class.
For the tiny atoms, this contributes the term $\sum_A \binom{m(A,S)}{r}.$
For the normal atoms, the number of forbidden \(r\)-sets is minimized when the \(k - m(S)\) vertices are distributed as evenly as possible among the \(\ell(S)\) normal atoms. 
Writing
$k-m(S)=(\ell(S)-q)\Bigl\lfloor\frac{k-m(S)}{\ell(S)}\Bigr\rfloor
+q\Bigl\lceil\frac{k-m(S)}{\ell(S)}\Bigr\rceil,$
the resulting contribution is
$(\ell(S)-q)\dbinom{\Bigl\lfloor\frac{k-m(S)}{\ell(S)}\Bigr\rfloor}{r}
+q\dbinom{\Bigl\lceil\frac{k-m(S)}{\ell(S)}\Bigr\rceil}{r}.$
Subtracting these non-free \(r\)-sets from the total \(\binom{k}{r}\) completes the proof for \(\ell(S)\geq 1\).

If \(\ell(S) = 0\), then the small branch of \(S\) is a single tiny atom (otherwise it would violate the assumption that no separator has two tiny atoms directly below it), so \(f(S) = 0\).
\end{proof}

\section{The deleting operation}\label{delete}
In this section, we analyze the decrease in the number of edges when all tiny atoms are deleted from an abstract tree. The reason we use the abstract tree framework is that we cannot delete atoms directly from the separator tree while ignoring the detailed graph structure within separator nodes and subgraph nodes. The deletion proceeds as follows.

\begin{definition}[Deletion Operation]
If $A$ is a tiny atom of $\widehat{T}$, we delete the subgraph node $A$. Let $S$ be the separator immediately above $A$; by definition, the branch of $S$ that does not contain $A$ cannot be a single tiny atom. Denote by $Q$ the subgraph node directly below $S$ other than $A$, and by $P$ the subgraph node directly above $S$ in $\widehat{T}$. We delete the separator node $S$ and merge $P$ and $Q$ into a single subgraph node. We emphasize that the resulting abstract tree is normal.
\end{definition}

\begin{remark}
Note that this deletion operation is well-defined.
The resulting tree \(\widehat{T'}\) is indeed an abstract tree of some real hypergraph. 
Suppose that \(a_1, \ldots, a_m\) are the sizes of all atoms in \(\widehat{T'}\). 
Let \(W_1, \ldots, W_m, S\) be disjoint vertex sets of sizes \(a_1 - k, \ldots, a_m - k, k\), respectively, and let \(A_i = W_i \cup S\).
Let \(H'\) be a hypergraph such that \(V(H') = \bigcup_{i=1}^m A_i\) and
\(
E(H') = \left( \bigcup_{i=1}^m \binom{A_i}{r} \right) \cup \left\{ e \in \binom{V(H')}{r} : e \cap S \neq \emptyset \right\}.
\)
Then \(\widehat{T'}\) is an abstract tree of \(H'\).
\end{remark}

\begin{figure}[htbp]
\centering
\begin{tikzpicture}[scale=0.9, node distance=1.2cm, every node/.style={font=\small}]
    \begin{scope}[local bounding box=left]
        \node[circle, minimum size=1.8mm, draw=black, fill=black, line width=0.6pt] (v1) at (0,1.2) {};
        \node[circle, minimum size=1.8mm, draw=black, fill=black, line width=0.6pt] (v2) at (-0.6,0) {};
        \node[circle, minimum size=1.8mm, draw=black, fill=black, line width=0.6pt] (v3) at (0.6,0) {};
        \node[circle, minimum size=1.8mm, draw=black, fill=black, line width=0.6pt] (v4) at (0.4,-0.9) {};
        \node[circle, minimum size=1.8mm, draw=black, fill=black, line width=0.6pt] (v5) at (1.4,-0.9) {};
        
        \node[circle, minimum size=1.8mm, draw=blue!80!black, fill=blue!80!black, line width=0.6pt] (w1) at (0,0.6) {};
        \node[circle, minimum size=1.8mm, draw=blue!80!black, fill=blue!80!black, line width=0.6pt] (w2) at (1,-0.5) {};
        
        \draw[black, line width=0.8pt] (v1) -- (w1);
        \draw[black, line width=0.6pt] (w1) -- (v2);
        \draw[black, line width=0.6pt] (w1) -- (v3);
        \draw[black, line width=0.6pt] (v3) -- (w2);
        \draw[black, line width=0.6pt] (w2) -- (v4);
        \draw[black, line width=0.6pt] (w2) -- (v5);
        
        \node[font=\small\bfseries] at (-0.4,1.4) {$H$};
        \node[font=\small\bfseries] at (-1,0) {$A$};
        \node[font=\small\bfseries] at (1.3,0) {$B \cup C$};
        \node[font=\small\bfseries] at (0.4,-1.25) {$B$};
        \node[font=\small\bfseries] at (1.4,-1.25) {$C$};
        \node[font=\small\bfseries, color=blue!80!black] at (0.45,0.6) {$S_1$};
        \node[font=\small\bfseries, color=blue!80!black] at (1.45,-0.5) {$S_2$};
    \end{scope}
    
    \node[below, font=\footnotesize\itshape, align=center] at (left.south) 
        {(a) Before deleting $A$};
    
    \begin{scope}[xshift=6.5cm, local bounding box=right]
        \node[circle, minimum size=1.8mm, draw=black, fill=black, line width=0.6pt] (v3) at (0,0.9) {};
        \node[circle, minimum size=1.8mm, draw=black, fill=black, line width=0.6pt] (v4) at (-0.8,-0.9) {};
        \node[circle, minimum size=1.8mm, draw=black, fill=black, line width=0.6pt] (v5) at (1.2,-0.9) {};
        
        \node[circle, minimum size=1.8mm, draw=blue!80!black, fill=blue!80!black, line width=0.6pt] (w2) at (0.5,0) {};
        
        \draw[black, line width=0.6pt] (v3) -- (w2);
        \draw[black, line width=0.6pt] (w2) -- (v4);
        \draw[black, line width=0.6pt] (w2) -- (v5);
        
        \node[font=\small\bfseries] at (0,1.3) {$B \cup C$ $(H)$};
        \node[font=\small\bfseries] at (-0.8,-1.3) {$B$};
        \node[font=\small\bfseries] at (1.2,-1.3) {$C$};
        \node[font=\small\bfseries, color=blue!80!black] at (0.8,0.3) {$S_2$};
    \end{scope}
    
    \node[below, font=\footnotesize\itshape, align=center] at ([xshift=6.5cm]left.south) 
        {(b) After deleting $A$};

\end{tikzpicture}

\caption{The small branch of $S_1$ contains only a single tiny atom $A$. After deleting $A$, the separator $S_1$ disappears, and the two subgraph nodes that were directly above and below $S_1$—namely $H$ (above) and $B\cup C$ (below)—merge into a single subgraph node. }
\label{fig:separator_tree_decomposition}
\end{figure}
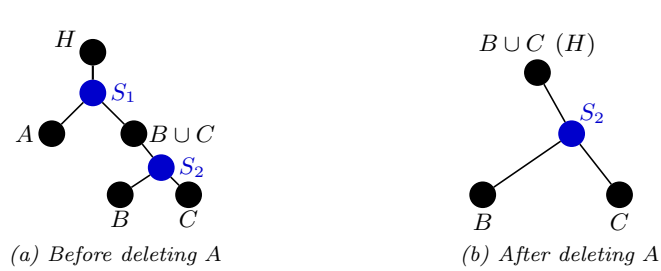

Performing this operation for every tiny atom of \(\widehat{T}\) yields a structure that can be verified to be again an abstract tree; moreover it is a normal abstract tree. We denote it by \(\widehat{T'}\). It is straightforward to verify the following description of the separator nodes in \(\widehat{T'}\).

\begin{itemize}
    \item If \(S\) is a separator node in \(\widehat{T}\) with reach \(0\), then \(S\) disappears after the deletion.
    \item If \(S\) is a separator node in \(\widehat{T}\) with reach at least \(1\), then \(S\) retains the same reach in \(\widehat{T'}\).
\end{itemize}

Next we analyze the change in the edge estimate caused by the deletion operation. 
Suppose that \(H \in \mathcal{F}_{c,r}(n,k)\) is a hypergraph with \(\widehat{T_H} = \widehat{T}\).
Let \(t+k\) and \(t'+k\) denote the numbers of vertices of \(\widehat{T}\) and \(\widehat{T'}\), respectively.
Let \(\mathcal{A}, \mathcal{A}'\) be the sets of atoms in \(\widehat{T}\) and \(\widehat{T'}\), respectively, and let \(\mathcal{S}, \mathcal{S}'\) be the sets of separators in \(\widehat{T}\) and \(\widehat{T'}\), respectively.
For each separator node \(S \in \mathcal{S}\), we define
\[
\widehat{f}(S) = \binom{k}{r} 
- \bigl(\ell(S) - q\bigr) \cdot \binom{\bigl\lfloor \frac{k}{\ell(S)} \bigr\rfloor}{r}
- q \cdot \binom{\bigl\lceil \frac{k}{\ell(S)} \bigr\rceil}{r},
\]
where \(q= k \bmod \ell(S)\).
Using the proof of Theorem~\ref{normal2}, it holds that
\begin{equation}\label{eq:after_deletion}
e(\widehat{T'}) \geq \binom{t'+k}{r} - \binom{t'}{r} + \sum_{A \in \mathcal{A}'} \binom{|A| - k}{r} + \sum_{S \in \mathcal{S}'} \widehat{f}(S).
\end{equation}
Then we consider the hypergraph \(H\) before deletion. 
By Lemma~\ref{lem:edge_f(S)}, 
\begin{equation}\label{eq:before_deletion}
e(H) \leq \binom{t+k}{r} - \binom{t}{r} + \sum_{A \in \mathcal{A}} \binom{|A| - k}{r} + \sum_{S \in \mathcal{S}} f(S) =  \binom{t+k}{r} - \binom{t}{r} + \sum_{A \in \mathcal{A}} \binom{|A| - k}{r} + \sum_{S \in \mathcal{S}'} f(S),
\end{equation}
where the equality holds because when \(S\in \mathcal{S}\setminus \mathcal{S}'\), \(\ell(S)=0\) and hence \(f(S)=0\).
Suppose that \(\mathcal{T}=\mathcal{A} \setminus \mathcal{A}' = \{A_1, A_2, \ldots, A_s\}\) be the set of tiny atoms and \(|A_i| = a_i + k\) for \(i = 1, 2, \ldots, s\).
Taking the difference between \eqref{eq:before_deletion} and \eqref{eq:after_deletion}, we obtain
\begin{equation}\label{eq:difference_deletion}
e(H) \leq e(\widehat{T'}) + \binom{t+k}{r} - \binom{t}{r} - \binom{t'+k}{r} + \binom{t'}{r} + \sum_{i=1}^s \binom{a_i}{r} + \sum_{S \in \mathcal{S}'} \bigl(f(S) - \widehat{f}(S)\bigr).
\end{equation}
We then reduce the problem of estimating the number of edges in a general hypergraph to estimating the number of edges in hypergraphs with a normal separator tree, together with the discrepancy between a general separator tree and a normal one.

\begin{definition}\label{essential}
We define the \emph{essential difference} of \(\widehat{T}\) as
\[
\mathcal{E}(\widehat{T}) = \binom{t+k}{r} - \binom{t}{r} - \binom{t - \sum_{i=1}^{s} a_i+k}{r} + \binom{t - \sum_{i=1}^{s} a_i}{r} + \sum_{i=1}^{s} \binom{a_i}{r},
\]
and for each \(S \in \mathcal{S}\), define the \emph{free difference} \(D_{f(S)} = f(S) - \widehat{f}(S)\). 
Note that \(t - t' = a_1 + a_2 + \cdots + a_s\). Consequently,
\begin{equation}\label{eq:delete_difference}
e(H) \leq e(\widehat{T'}) + \mathcal{E}(\widehat{T}) + \sum_{S \in \mathcal{S}'} D_{f(S)}.    
\end{equation}
\end{definition}

For \(\ell(S)=1\) the difference can be estimated directly. For \(\ell(S)\ge 2\) we give a uniform upper bound.

\begin{proposition}\label{free2}
For any separator node \(S\) with \(\ell(S) \ge 2\),
\begin{equation}\label{eq:dfs_bound}
D_{f(S)} \le \frac{m(S) r}{2^{r-1} k} \binom{k}{r}.
\end{equation}
\end{proposition}

\begin{proof}
Define the function 
\[
g(s, \ell) = (\ell - q) \binom{\bigl\lfloor \frac{s}{\ell} \bigr\rfloor}{r} + q \binom{\bigl\lceil \frac{s}{\ell} \bigr\rceil}{r},
\qquad \text{where } q = s \bmod \ell.
\]
Fix a separator node \(S\). Let \(m = m(S) = \sum_A m(A,S)\) and \(\ell_0 = \ell(S)\).

Then, by Lemma~\ref{technical data}, the quantity \(f(S)\) is bounded above by 
\[
f(S) \le \binom{k}{r} - g(k - m, \ell_0) - \sum_{A \in \mathcal{T}} \binom{m(A,S)}{r}.
\]
By the definition of $\widehat{f}(S)$, we have
$\widehat{f}(S) = \binom{k}{r} - g(k, \ell_0)$,
hence 
\begin{equation}\label{eq:dfs_bound2}
D_{f(S)} = f(S) - \widehat{f}(S) \le g(k, \ell_0) - g(k - m, \ell_0) - \sum_{A \in \mathcal{T}} \binom{m(A,S)}{r}.
\end{equation}

Since the summation term is nonnegative, to establish \eqref{eq:dfs_bound} it suffices to show that for all \(\ell_0 \ge 2\),
\[
g(k, \ell_0) - g(k - m, \ell_0) \le \frac{m r}{2^{r-1} k} \binom{k}{r}.
\]

Observe that for any \(s\),
\[
g(s+1, \ell_0) - g(s, \ell_0) = \binom{\bigl\lfloor \frac{s}{\ell_0} \bigr\rfloor}{r-1} \le \binom{\bigl\lfloor \frac{s}{2} \bigr\rfloor}{r-1} \le \frac{1}{2^{r-1}} \binom{s}{r-1}.
\]
Summing this inequality over \(s = k - m, k - m + 1, \ldots, k - 1\) yields
\[
g(k, \ell_0) - g(k - m, \ell_0) \le \frac{1}{2^{r-1}} \sum_{s = k - m}^{k - 1} \binom{s}{r-1}.
\]
Since \(\binom{s}{r-1}\) is increasing in \(s\) for fixed \(r\), we have \(\binom{s}{r-1} \le \binom{k}{r-1}\) for all \(s \le k\). Therefore,
\[
\sum_{s = k - m}^{k - 1} \binom{s}{r-1} \le m \binom{k}{r-1} = m \cdot \frac{r}{k} \binom{k}{r}.
\]
Combining these inequalities gives the desired bound.
\end{proof}

The point of this proposition is that, for a separator $S$, when the reach \(\ell(S)\) is large, the contribution of free difference of $S$ to the overall missing‑edge count becomes negligible.

\section{Upper bounds}\label{bound}
In this section, we derive an upper bound for the edge number of an arbitrary abstract tree.
\begin{theorem}\label{tight}
Let $\widehat{T}$ be the associated abstract tree of a separator tree $T$ with $t+k$ vertices and $r \ge 3$
\begin{enumerate}
    \item If $t \ge \frac{2^{r}}{2^{r}-1}ck$ and $k \ge r $, then
    \begin{equation}\label{eq:general_bound}
     e(T)\leq \binom{t+k}{r} - \binom{t}{r} 
            + \frac{t}{ck}\binom{ck}{r}
            + \Bigl(\frac{3t}{2ck}-1\Bigr)\binom{k}{r}
            - \frac{t}{2ck} \sum_{i = 0}^{\infty} \binom{k / 2^{\,i}}{r}.    
    \end{equation}
    \item For $t\geq k+\frac{4}{3}ck+\frac{ck^2}{r-1}$, if $c\geq (2r)^{1/(r-1)}$ and $k \ge 2(r-1)$, then
    \begin{equation}\label{eq:special_bound}
    e(T)\leq \binom{t+k}{r} - \binom{t}{r} 
            + \frac{t}{ck}\binom{ck}{r}
            + \Bigl(\frac{t}{ck}-1\Bigr)\binom{k}{r}
            - \frac{t}{2ck} \sum_{i = 0}^{\infty} \binom{k / 2^{\,i}}{r}.    
    \end{equation}
\end{enumerate}
\end{theorem}

\begin{proof}
If \(L(T)=0\), then \(T\) consists of a tiny atom, hence \(t\leq 1+\frac{k}{3} < \frac{2^r}{2^r-1}ck\), violates the condition of Theorem~\ref{tight}.
Therefore, we assume that \(L(T)\geq 1\).

Let \(\mathcal{A}, \mathcal{S}, \mathcal{T}\) denote the sets of atoms, separators and tiny atoms of \(T\), respectively.
Performing the deletion operation on \(\widehat{T}\) yields a normal abstract tree \(\widehat{T'}\).
Suppose that \(\widehat{T'}\) has \(t' + k\) vertices; then the number of deleted vertices is \(t - t'\).
Let \(\mathcal{A}'\) and \(\mathcal{S}'\) denote the sets of atoms and separators of \(\widehat{T'}\), respectively.
By \eqref{eq:normal_abstract_def} we have
\[
e(\widehat{T'}) = \binom{t' + k}{r} - \binom{t'}{r} + \frac{t'}{ck}\binom{ck}{r} + \left( \frac{t'}{ck} - 1 \right) \binom{k}{r} - \frac{t'}{2ck} \sum_{i=0}^{\infty} \binom{k / 2^{i}}{r} + f(k, r, L) - X(\widehat{T'}).
\]

Where $L=L(T)=L(T')$, by Definition~\ref{essential}, we have
\[
e(T) \le e(\widehat{T'}) + \mathcal{E}(\widehat{T}) + \sum_{S \in \mathcal{S}'} D_{f(S)},
\]
where
\[
\mathcal{E}(\widehat{T}) = \binom{t + k}{r} - \binom{t}{r} - \binom{t' + k}{r} + \binom{t'}{r} + \sum_{A \in \mathcal{T}} \binom{|A| - k}{r}.
\]
Combining these yields
\[
\begin{aligned}
e(T)&\leq \binom{t+k}{r}-\binom{t}{r} + \frac{t'}{ck}\binom{ck}{r} + \left( \frac{t'}{ck} - 1 \right) \binom{k}{r}
- \frac{t'}{2ck} \sum_{i=0}^{\infty} \binom{k/2^i}{r}\\
& \qquad+f(k, r, L)-X(\widehat{T'}) + \sum_{A\in\mathcal{T}}\binom{|A|-k}{r} + \sum_{S\in\mathcal{S}'} D_{f(S)}.
\end{aligned}
\]

Note that \(X(T) = \sum_{S \in \mathcal{S}} X(S)\).
Compared with \eqref{eq:general_bound} and \eqref{eq:special_bound}, and note that \(\sum_{i=0}^\infty\binom{k/2^i}{r}\leq \sum_{i=0}^\infty\frac{1}{2^{ir}}\binom{k}{r}= \frac{2^r}{2^r-1}\binom{k}{r} \), it suffices to prove
\begin{enumerate}
\item For \(t \ge \frac{2^{r}}{2^{r} - 1} ck\) and \(k \ge r\),
\begin{equation}\label{eq:general_bound2}
f(k, r, L) + \sum_{A \in \mathcal{T}} \binom{|A| - k}{r} + \sum_{S \in \mathcal{S}'} \bigl( D_{f(S)} - X(S) \bigr)
\le \frac{t - t'}{ck} \binom{ck}{r}  + \frac{t}{2ck} \binom{k}{r} + \frac{t - t'}{ck}\cdot\left(1- \frac{2^{r-1}}{2^r-1}\right)\binom{k}{r}
\end{equation}
\item For \(c \ge (2r)^{1/(r-1)}\), \(k \ge 2(r-1)\) and \(t\geq k+\frac{4}{3}ck+\frac{ck^2}{r-1}\),
\begin{equation}\label{eq:special_bound2}
f(k, r, L) + \sum_{A \in \mathcal{T}} \binom{|A| - k}{r} + \sum_{S \in \mathcal{S}'} \bigl( D_{f(S)} - X(S) \bigr)
\le \frac{t - t'}{ck}  \binom{ck}{r} + \frac{t - t'}{ck}\cdot\left(1- \frac{2^{r-1}}{2^r-1}\right)\binom{k}{r}
\end{equation}
\end{enumerate}

The main part of this proof is estimating \(f(k,r,L)+\sum_{A \in \mathcal{T}} \binom{|A| - k}{r} + \sum_{S \in \mathcal{S}'} \bigl( D_{f(S)} - X(S) \bigr)\).

Let \(V_{1}\) be the set of tiny vertices in \(T\) that are assigned to some separator \(S\) with \(\ell(S) = 1\), and let \(V_{2}\) be the set of all other tiny vertices.
Let \(\mathcal{S}_1 = \{S \in \mathcal{S} : \ell(S) = 1,\ \exists\text{ tiny vertex assigned to } S\}\) and \(\mathcal{S}_2 = \{S \in \mathcal{S} : \ell(S) \ge 2,\ \exists\text{ tiny vertex assigned to } S\}\).
Then clearly,
\[
|V_1| = \sum_{S\in\mathcal{S}_1} m(S),\quad |V_1|+|V_2|=t-t' \quad\text{and}\quad |V_2|\ge\sum_{S\in\mathcal{S}_2} m(S).
\]
(Note that there may exist tiny vertices not assigned to any separator, and there may exist separator $S$ with no tiny vertex assigned to $S$; such $S$ satisfy $D_f(S) \le 0$.)

\medskip
\noindent \textbf{Free edges in \(S \in \mathcal{S}_2\).} First, we focus on the contribution from tiny vertices assigned to separators of reach at least $2$.
By Proposition~\ref{free2},
\begin{equation}\label{eq:et2_1}
\sum_{S \in \mathcal{S}_2} \bigl( D_{f(S)} - X(S) \bigr)
\le \sum_{S \in \mathcal{S}_2} D_{f(S)}
\le \sum_{S \in \mathcal{S}_2} \frac{m(S) r}{2^{r-1} k} \binom{k}{r}
\le \frac{|V_2| r}{2^{r-1} k} \binom{k}{r}.    
\end{equation}

For each \(A \in \mathcal{T}\), let \(A' = A \setminus V_2\) and let \(\mathcal{T}' = \{A' : A \in \mathcal{T}\}\). Then
\begin{equation}\label{eq:et2_2}
\sum_{A \in \mathcal{T}} \binom{|A| - k}{r} - \sum_{A' \in \mathcal{T}'} \binom{|A'| - k}{r}
\le \sum_{A \in \mathcal{T}} |A \cap V_2| \binom{k/3-1}{r-1}
\le \frac{|V_2| r}{3^{r-1} k} \binom{k}{r}.
\end{equation}

Note that for $S \in \mathcal{S}' \setminus (\mathcal{S}_1 \cup \mathcal{S}_2)$, $D_{f(S)} \le 0$. 
From \eqref{eq:et2_1} and \eqref{eq:et2_2} we obtain
\begin{equation}\label{eq:et2}
\sum_{A \in \mathcal{T}} \binom{|A| - k}{r} + \sum_{S \in \mathcal{S}'} \bigl( D_{f(S)} - X(S) \bigr)
\le \sum_{A' \in \mathcal{T}'} \binom{|A'| - k}{r} + \sum_{S \in \mathcal{S}_1} \bigl( D_{f(S)} - X(S) \bigr)
+ \left( \frac{1}{2^{r-1}} + \frac{1}{3^{r-1}} \right) \frac{|V_2| r}{k} \binom{k}{r}.
\end{equation}

\medskip
\noindent \textbf{Free edges in \(S \in \mathcal{S}_1\).}
Suppose \(A\) is a tiny atom. Then there is at most one separator \(S \in \mathcal{S}_1\) such that \(A\) lies in the small branch of \(S\).
Hence the tiny vertices in \(A' = A \setminus V_2\) are exactly the set of vertices in \(A\) that are assigned to \(S\).
For any \(S \in \mathcal{S}_1\), let \(\mathcal{A}_S\) be the set of tiny atoms in the small branch of \(S\). Then by Lemma~\ref{technical data},
\[
D_{f(S)} = f(S) \le \binom{k}{r} - \sum_{A \in \mathcal{A}_S} \binom{m(A,S)}{r} - \binom{k - m(S)}{r}.
\]
Note that 
\[
\sum_{S \in \mathcal{S}_1}\sum_{A \in \mathcal{A}_S}\binom{m(A,S)}{r}=\sum_{A' \in \mathcal{T}'} \binom{|A'| - k}{r}.
\]
This yields
\[
\sum_{A' \in \mathcal{T}'} \binom{|A'| - k}{r} + \sum_{S \in \mathcal{S}_1} \bigl( D_{f(S)} - X(S) \bigr)
\le \sum_{S \in \mathcal{S}_1} \left( \binom{k}{r} - \binom{k - m(S)}{r} -X(S)\right).
\]
If \(\mathcal{S}_1\) is nonempty, then let \(x = \frac{|V_1|}{k\cdot|\mathcal{S}_1|}\), by convexity of \(h(m)=\binom{k-m}{r}\), it holds that
\begin{equation}\label{eq:et1_1}
\sum_{A' \in \mathcal{T}'} \binom{|A'| - k}{r} + \sum_{S \in \mathcal{S}_1} \bigl( D_{f(S)} - X(S) \bigr)
\le |\mathcal{S}_1| \binom{k}{r} - |\mathcal{S}_1|\binom{k(1-x)}{r} -\sum_{S\in\mathcal{S}_1}X(S).
\end{equation}

Let \(R\) be the set of balanced separators \(S\) with \(\ell(S) = 1\) and let \(P\) be the set of unbalanced separators \(S\) with \(\ell(S) = 1\). 
Clearly \(|\mathcal{S}_1|=|R|+|P|\).
Moreover, \(R\) and \(P\) satisfy the following constraint.

\begin{claim}\label{claim2}
If \(P\) is not empty, then for any \(S\in P\)
\[
X(S)\geq \frac{1}{2}\binom{k}{r}-\frac{3}{2}\binom{k/2}{r}.
\]
\end{claim}

We prove Claim~\ref{claim2} in the end of Section~\ref{bound}.
Assume Claim~\ref{claim2} holds, the right hand side of \eqref{eq:et1_1} reduces to
\begin{equation*}
\Bigl(|R|+\frac{1}{2}|P|\Bigr)\binom{k}{r} + \frac{3}{2}|P|\binom{k/2}{r} - \bigl(|R|+|P|\bigr)\binom{k(1-x)}{r}.
\end{equation*}
Combining this with \eqref{eq:et2} yields
\begin{equation}\label{eq:et1}
\begin{aligned}
&\ f(k, r, L) + \sum_{A \in \mathcal{T}} \binom{|A| - k}{r} + \sum_{S \in \mathcal{S}'} \bigl( D_{f(S)} - X(S) \bigr)\\
\leq&\ f(k, r, L)+\left( \frac{1}{2^{r-1}} + \frac{1}{3^{r-1}} \right) \frac{|V_2| r}{k} \binom{k}{r} + \Bigl(|R|+\frac{1}{2}|P|\Bigr)\binom{k}{r} + \frac{3}{2}|P|\binom{k/2}{r} - \bigl(|R|+|P|\bigr)\binom{k(1-x)}{r}.
\end{aligned}
\end{equation}

\bigskip
\noindent\textbf{Proof of \eqref{eq:general_bound2}.}
We consider three cases: (a) \(R,P = \emptyset\); (b) \(P\neq\emptyset\) or \(R\neq \emptyset\).

\medskip
\noindent\textbf{Case (a).} Suppose \(R = P = \emptyset\). Then \(|V_1| = 0\), \(|V_2| = t - t'\), and the right-hand side of \eqref{eq:et1} reduces to
\(f(k,r,L) + \bigl( \frac{1}{2^{r-1}} + \frac{1}{3^{r-1}} \bigr) \frac{|V_2| r}{k} \binom{k}{r}\).
Note that \(f(k,r,\ell)\) is decreasing in \(\ell\) (Claim~\ref{convex}, since \(f(k,r,\ell)=g(\ell/k)=\sum_{i=0}^\infty g_i(\ell/k)\) and each \(g_i(z)\) is decreasing in \(z\)). Hence
\[
f(k,r,L) \le f(k,r,1) = \frac12\sum_{i=0}^\infty \binom{k/2^i}{r}
\le \frac{2^{r-1}}{2^r-1}\binom{k}{r}
\le \frac{t}{2ck}\binom{k}{r},
\]
where the last inequality uses \(t \ge \frac{2^r}{2^r-1}ck\).
To prove \eqref{eq:general_bound2}, it suffices to show
\begin{equation}\label{eq:proof_of_v2}
\Bigl( \frac{1}{2^{r-1}} + \frac{1}{3^{r-1}} \Bigr) \frac{|V_2| r}{k}\binom{k}{r}
\le \frac{|V_2|}{ck}\binom{ck}{r} + \frac{|V_2|}{ck}\Bigl(1-\frac{2^{r-1}}{2^r-1}\Bigr)\binom{k}{r}.
\end{equation}
This reduces to
\[
c^r - \Bigl(\frac{1}{2^{r-1}}+\frac{1}{3^{r-1}}\Bigr) c r + 1 - \frac{2^{r-1}}{2^r-1} \ge 0
\]
for \(r\ge 3\) and \(c\ge 1\). This inequality holds because
\[
c^r - \Bigl(\frac{1}{2^{r-1}}+\frac{1}{3^{r-1}}\Bigr) c r + 1 - \frac{2^{r-1}}{2^r-1}
\ge c^r - \frac{13}{12}c + \frac37 \ge 0.
\]

\medskip
\noindent\textbf{Case (b).} Suppose \(P\neq\emptyset\) or \(R\neq\emptyset\). Using \eqref{eq:et1}, \eqref{eq:proof_of_v2} and \(t-t' = |V_1|+|V_2|\), it suffices to prove
\begin{equation}\label{eq:case_b1}
\begin{aligned}
& f(k,r,L) + \bigl(|R|+\tfrac12|P|\bigr)\binom{k}{r} + \tfrac32|P|\binom{k/2}{r} - (|R|+|P|)\binom{k(1-x)}{r} \\
\le &\ \frac{|V_1|}{ck}\binom{ck}{r} + \frac{t}{2ck}\binom{k}{r} + \frac{|V_1|}{ck}\Bigl(1-\frac{2^{r-1}}{2^r-1}\Bigr)\binom{k}{r}.
\end{aligned}
\end{equation}

\begin{claim}\label{claim1}
\(2ck|R| + ck|P| \le t\). 
\end{claim}

We prove Claim~\ref{claim1} at the end of Section~\ref{bound}.

Assume Claim~\ref{claim1} holds, \(|R|+\frac12|P| \le \frac{t}{2ck}\). Since there exists a separator \(S\) with \(\ell(S)=1\), we have \(L \ge 2\ell(S) \ge 2\). Consequently,
\[
f(k,r,L) \le f(k,r,2) \le \sum_{i=0}^\infty \binom{k/2^{i+1}}{r}
\le \frac{2^r}{2^r-1}\binom{k/2}{r} \le \frac87\binom{k/2}{r},
\]
and thus
\[
\frac32|P|\binom{k/2}{r} + f(k,r,L)
\le \frac32|P|\binom{k/2}{r} + \frac87\binom{k/2}{r}
\le \frac{27}{8}(|R|+|P|)\binom{k/2}{r}
\le (|R|+|P|)\binom{2k/3}{r}.
\]
To establish \eqref{eq:case_b1}, it remains to verify
\begin{equation}\label{eq:case_b2}
\begin{aligned}
& (|R|+|P|)\binom{k(1-x)}{r}
+ \frac{|V_1|}{ck}\binom{ck}{r} + \frac{|V_1|}{ck}\Bigl(1-\frac{2^{r-1}}{2^r-1}\Bigr)\binom{k}{r} 
 - (|R|+|P|)\binom{2k/3}{r} \ge 0.
\end{aligned}
\end{equation}
Here \(x = \dfrac{|V_1|}{k|\mathcal{S}_1|} = \dfrac{|V_1|}{k(|R|+|P|)}\).
If \(x \le \frac13\), then \(\binom{k(1-x)}{r} \ge \binom{2k/3}{r}\) and \eqref{eq:case_b2} holds.
If \(x > \frac13\), then \(|V_1| \ge \frac13 k(|R|+|P|)\), and it suffices to show
\[
\frac{1}{3c}\binom{ck}{r} + \frac{1}{3c}\Bigl(1-\frac{2^{r-1}}{2^r-1}\Bigr)\binom{k}{r} - \binom{2k/3}{r} \ge 0.
\]
This follows from
\[
c^r + 1 - \frac{2^{r-1}}{2^r-1} - \frac{2^r}{3^{r-1}}c \ge c^r - \frac89c + \frac37 \ge 0.
\]
This completes Case (b) and the proof of \eqref{eq:general_bound2}.

\bigskip
\noindent\textbf{Proof of \eqref{eq:special_bound2}.}
Since \(x = \frac{|V_1|}{k|\mathcal{S}_1|}\) and \(k \ge 2(r-1)\),
\[
|\mathcal{S}_1|\binom{k}{r} - |\mathcal{S}_1|\binom{k(1-x)}{r}
\le kx|\mathcal{S}_1|\binom{k}{r-1}
\le |V_1|\frac{r}{k-r+1}\binom{k}{r}
\le |V_1|\frac{2r}{k}\binom{k}{r}.
\]
Combining with \eqref{eq:et2} and \eqref{eq:et1_1}, we obtain
\[
\sum_{A \in \mathcal{T}} \binom{|A| - k}{r} + \sum_{S \in \mathcal{S}'} \bigl( D_{f(S)} - X(S) \bigr)
\le |V_1|\frac{2r}{k}\binom{k}{r} + \left(\frac{1}{2^{r-1}}+\frac{1}{3^{r-1}}\right)\frac{|V_2|r}{k}\binom{k}{r}
\le \frac{2r(t-t')}{k}\binom{k}{r}.
\]

To establish \eqref{eq:special_bound2}, it suffices to show
\[
f(k,r,L) \le \frac{t-t'}{ck}\binom{ck}{r} - \frac{2r(t-t')}{k}\binom{k}{r}
+ \frac{t-t'}{ck}\left(1 - \frac{2^{r-1}}{2^r-1}\right)\binom{k}{r}.
\]
Using \(c \ge (2r)^{1/(r-1)}\) and \(r \ge 3\), we have
\(\frac{1}{c}\binom{ck}{r} \ge c^{r-1}\binom{k}{r} \ge 2r\binom{k}{r}\) and \(1 - \frac{2^{r-1}}{2^r-1} \ge \frac{3}{7}\).
Hence it suffices to show
\begin{equation}\label{eq:special1}
f(k,r,L) \le \frac{3}{7c}\cdot\frac{t-t'}{k}\binom{k}{r}.
\end{equation}

If \(L \le \frac{k}{r-1}\), then \(\widehat{T'}\) contains \(L\) normal atoms, so
\(t' = |V(\widehat{T})| \le k + ck \cdot L \le k + \frac{ck^2}{r-1}\).
In this case, \(f(k,r,L) \le f(k,r,1) \le \frac{2^{r-1}}{2^r-1}\binom{k}{r} \le \frac{4}{7}\binom{k}{r}\).
Since \(t \ge k + \frac{4}{3}ck + \frac{ck^2}{r-1}\), we have \(\frac{3}{7c}\cdot\frac{t-t'}{k} \ge \frac{4}{7}\), which implies \eqref{eq:special1} and hence \eqref{eq:special_bound2} holds.

If \(L > \frac{k}{r-1}\), then
\[
f(k,r,L) = \frac{L}{2}\sum_{i=0}^\infty \binom{k/(2^i L)}{r} = 0,
\]
which directly gives \eqref{eq:special1} and completes the proof.

\end{proof}

\begin{proof}[Proof of Claim~\ref{claim2}]
Recall that
\[
X(S) = f(k,r,\ell(S)+\ell(S,+))-f(k,r,\ell(S))-f(k,r,\ell(S,+))
      +\ell(S)\,\binom{\frac{k}{\ell(S)}}{r}.
\]
For $S \in P$, we have $\ell(S)=1$ and $\ell(S,+)\ge 2$. By Claim~\ref{convex}, $f(k,r,\ell)=g(\ell/k)$, and $\tilde{g}(y)=g(y)-g(1/k+y)$ is decreasing on $\mathbb{R}^+$. Hence
\[
X(S)=\frac{1}{2}\binom{k}{r}-\frac{1}{2}\sum_{i=1}^{\infty}\binom{k/2^i}{r}
      -g(\ell(S,+)/k)+g((\ell(S,+)+1)/k)
      \ge \frac{1}{2}\binom{k}{r}-\frac{1}{2}\sum_{i=1}^{\infty}\binom{k/2^i}{r}-g(2/k)+g(3/k).
\]
Now
\[
-g(2/k)+g(3/k) 
= -\sum_{i=1}^{\infty}\binom{k/2^{i}}{r} + \frac{3}{2}\sum_{i=0}^{\infty}\binom{k/(3\cdot 2^{i})}{r}.
\]
Therefore,
\[
X(S) \ge \frac{1}{2}\binom{k}{r}-\frac{3}{2}\sum_{i=1}^{\infty}\binom{k/2^{i}}{r} + \frac{3}{2}\sum_{i=0}^{\infty}\binom{k/(3\cdot 2^{i})}{r}
\ge \frac{1}{2}\binom{k}{r} - \frac{3}{2}\binom{k/2}{r}.
\]
This completes the proof of the claim.
\end{proof}

\begin{proof}[Proof of Claim~\ref{claim1}]
We aim to prove
\[
2ck|R| + ck|P| \le t.
\]

For each \(S \in R\), let \(A_S\) and \(B_S\) be its small branch and big branch, respectively. For each \(S' \in P\), let \(P_{S'}\) be its small branch. 
Each of \(A_S\), \(B_S\) and \(P_{S'}\) contains exactly one normal atom, and there is no separator of reach \(1\) inside any of these branches. Consequently, the branches \(A_S\), \(B_S\) (for \(S \in R\)) and \(P_{S'}\) (for \(S' \in P\)) are pairwise disjoint. Therefore,
\[
\sum_{S \in R} \bigl( (|A_S| - k) + (|B_S| - k) \bigr) + \sum_{S' \in P} (|P_{S'}| - k) \le t.
\]

It remains to show that \(|A_S|, |B_S|, |P_{S'}| \ge ck + k\) for all \(S \in R\) and \(S' \in P\).

Consider some \(S \in R \cup P\). Since there exists a tiny vertex assigned to \(S\), its small branch must contain a normal atom and a tiny atom. Hence there are at least two atoms \(A_1, A_2\) in its small branch directly below a separator \(S_0\). If \(|A_1 \setminus S_0| + |A_2 \setminus S_0| \le ck\), then we could merge them into a single atom, contradicting our assumption. Thus \(|A_1 \setminus S_0| + |A_2 \setminus S_0| \ge ck\), which implies that the small branch of \(S\) has size at least \(ck + k\).

If \(S \in R\), then because its small branch contains tiny vertices and \(S\) is balanced, its big branch also contains a tiny atom. By the same argument, the big branch of \(S\) has size at least \(ck + k\). This completes the proof.
\end{proof}

\section{Proofs for main theorems}\label{proof}

In this section we prove Theorem~\ref{thm:limit}, as well as the upper bounds for Theorems~\ref{main} and~\ref{0}.

We first prove Theorem~\ref{thm:limit}.
The following lemma will be used.

\begin{lemma}\label{lem:limit_exist}
Suppose that \(\{a_n : n \in \mathbb{N}^*\}\) satisfies the following conditions:
\begin{itemize}
\item [(1)] there exists \(C > 0\) such that \(0 \le a_n \le C\) for all \(n \in \mathbb{N}^*\);
\item [(2)] \(\{n a_n\}\) is increasing.
\item [(3)] \(a_{mn} \ge a_n\) for all \(m, n \in \mathbb{N}^*\);
\end{itemize}
Then \(\lim_{n \to \infty} a_n\) exists.
\end{lemma}

\begin{proof}
First, \(A = \limsup_{n \to \infty} a_n\) exists because \(\{a_n\}\) has a uniform upper bound.
For any \(\varepsilon > 0\), there exists \(n \in \mathbb{N}\) such that \(a_n \ge A - \dfrac{\varepsilon}{2}\) and \(\dfrac{1}{n} \le \dfrac{\varepsilon}{2C}\).

Now consider any \(m \ge n^2\). Write \(m = qn + r\) with \(0 \le r \le n-1\). Then \(q \ge n\). Since \(\{n a_n\}\) is increasing, we have \(m a_m \ge qn a_{qn}\), and consequently
\[
a_m \ge \frac{qn}{m} \cdot a_{qn} \ge \left(1 - \frac{r}{qn + r}\right) a_n.
\]
Using \(r \le n-1\) and \(q \ge n\), we obtain \(1 - \frac{r}{qn + r} \ge 1 - \frac{1}{n}\). Hence
\[
a_m \ge \left(1 - \frac{1}{n}\right) a_n \ge a_n - \frac{a_n}{n} \ge a_n - \frac{C}{n} \ge A - \frac{\varepsilon}{2} - \frac{\varepsilon}{2} = A - \varepsilon.
\]
Thus for all sufficiently large \(m\), we have \(a_m \ge A - \varepsilon\), and since \(A = \limsup a_n\) implies \(a_m \le A + \varepsilon\) for large \(m\), we conclude \(\lim_{n \to \infty} a_n = A\).
\end{proof}

Now we are ready to present the proof of Theorem~\ref{thm:limit}.

\begin{proof} [\bf Proof of Theorem~\ref{thm:limit}]
For any fixed \(r\geq 2\), \(c\) and \(k\), our aim is to show that the sequence \(\{a_n := g_{c,r}(n + k, k)\}\) satisfies the three conditions of Lemma~\ref{lem:limit_exist}. If this holds, then the limit \(g_{c,r}^k := \lim_{n \to \infty} g_{c,r}(n, k)\) exists and equals \(\lim_{n \to \infty} a_n\).

First, we verify condition (1), i.e., there exists \(C > 0\) such that \(a_n \le C\). Recall that 
\[
f_{c,r}(n,k) = \max_{H \in \mathcal{F}_{c,r}(n,k)} e(H) - \binom{n}{r} + \binom{n - k}{r}, \qquad 
g_{c,r}(n,k) = \frac{r! \cdot f_{c,r}(n,k)}{k^{r-1}(n - k)}.
\]

Note that Lemma~\ref{lem:edgenum_septree} and the separator tree framework of Section~\ref{sep-tree} are valid for all \(r\ge 2\) (the bonded edge term \(\bar{\beta}(S)\) vanishes when \(r=2\), but the inequalities remain valid). 
For any \(H \in \mathcal{F}_{c,r}(n + k, k)\), by dropping the non-negative terms \(\bar{e}(A)\) and \(\bar{\beta}(S)\) while noting that \(\bar{e}(S) \le \binom{k}{r}\), \(\sum_A a_A = n\), and \(|\mathcal{S}| = |\mathcal{A}| - 1 \le n\), we obtain
\[
f_{c,r}(n + k, k)
\le \sum_{A \in \mathcal{A}} \binom{a_A}{r} + \sum_{S \in \mathcal{S}} \binom{k}{r}
\le n \binom{ck}{r} + n \binom{k}{r}.
\]
Consequently
\[
a_n = \frac{r!}{n k^{r-1}}\,f_{c,r}(n + k, k)
\le \frac{r!}{k^{r-1}} \Bigl( \binom{ck}{r} + \binom{k}{r} \Bigr),
\]
which is a constant depending only on \(c, k, r\) and satisfies condition (1).

Second, we verify condition (2). Note that \(n a_n\) is proportional to \(f_{c,r}(n + k, k)\); hence it suffices to show that \(f_{c,r}(n + k, k)\) is increasing in \(n\).

Let \(H \in \mathcal{F}_{c,r}(n + k, k)\) be a hypergraph that attains the maximum number of edges.
Then \(f_{c,r}(n+k,k) = e(H) - \binom{n+k}{r}+\binom{n}{r}\). 
Let \(S\) be a subset of \(V(H)\) of size \(k\) and let \(H'\) be the hypergraph obtained from \(H\) by adding a new vertex \(v\) together with all edges in 
$V(H) \cup \{ v  \}  $
that intersect both \(v\) and \(S\). Then \(H' \in \mathcal{F}_{c,r}(n + 1, k)\) and
\[
e(H') = e(H) + \binom{|V(H)\cup \{v\}|}{r} - \binom{|V(H)|}{r} - \binom{|(V(H)\cup \{v\})\setminus S|}{r} + \binom{|V(H)\setminus S|}{r}.
\]
Consequently,
\[
f_{c,r}(n+k+1,k)\geq e(H') -\binom{n+k+1}{r}+\binom{n+1}{r} = e(H) - \binom{n+k}{r}+\binom{n}{r} = f_{c,r}(n+k,k),
\]
which satisfies condition (2).

Next, we verify condition (3). It suffices to show that \(f_{c,r}(mn + k, k) \ge m \cdot f_{c,r}(n + k, k)\) for all \(m, n \in \mathbb{N}^*\).

Again, let \(H \in \mathcal{F}_{c,r}(n + k, k)\) be a hypergraph that attains the maximum number of edges and let \(S\) be a \(k\)-subset of \(V(H)\).
Construct a hypergraph \(H_m\) as follows:
\begin{itemize}
\item \(V=V(H_m) = V_1 \sqcup V_2 \sqcup \cdots \sqcup V_m \sqcup S\);
\item For each \(i = 1, 2, \ldots, m\), the induced subgraph \(H_m[V_i \cup S]\) is a copy of \(H\);
\item For any \(r\)-subset \(e\) that intersects at least two distinct \(V_i\)'s, \(e\in E(H_m)\) if and only if it intersects \(S\).
\end{itemize}
Then \(H_m\in\mathcal{F}_{c,r}(mn+k,k)\) and
\begin{equation}\label{eq:edge_gluing}
e(H_m) = m\cdot e(H) - (m-1)\cdot e(S) + \binom{|V|}{r} - \binom{|V\setminus S|}{r} -\binom{|S|}{r} - \sum_{i=1}^m\left( \binom{|V_i\cup S|}{r} - \binom{|V_i|}{r}-\binom{|S|}{r}\right).
\end{equation}
This yields
\[
f_{c,r}(mn+k,k)\geq e(H_m) -\binom{mn+k}{r}+\binom{mn}{r} = m\cdot\left( e(H)-\binom{n+k}{r}+\binom{n}{r}\right)+(m-1)\left(\binom{k}{r}-e(S)\right).
\]
Note that \(e(S)\leq \binom{k}{r}\). Consequently,
\[
f_{c,r}(mn+k,k)\geq m \cdot f_{c,r}(n+k,k),
\]
which satisfies condition (3) and completes the proof.
\end{proof}

Finally, we establish Theorem~\ref{main} and the upper bound in Theorem~\ref{0}.

\begin{proof}[\bf Proof of Theorem~\ref{main} and the upper bound in Theorem~\ref{0}]
Let \(H\) be an \(n\)-vertex \(r\)-uniform hypergraph with no \((k+1)\)-connected subgraph, and let \(T_H\) be a separator tree of \(H\) with associated abstract tree \(\widehat{T_H}\).
By Corollary~\ref{cor:edge_abstract_tree}, we have \(e(H) \le e(\widehat{T_H})\), while
Theorem~\ref{tight} provides the two upper bounds for \(e(\widehat{T_H})\) in the two stated cases.
In the first case of Theorem~\ref{tight}, for \(n-k \ge \frac{2^{r}}{2^{r}-1}ck\) and \(k \ge r\), we obtain
\[
\begin{aligned}
g_{c,r}(n,k) 
&\le \frac{r!}{c k^{r}}\Bigl(\binom{ck}{r}+\binom{k}{r}-\frac{1}{2}\sum_{i=1}^{\infty}\binom{k/2^i}{r}\Bigr)-\frac{r!}{(n-k) k^{r-1}}\binom{k}{r}\\
&\le c^{r-1}+\frac{1}{2c}+\frac{1}{2c}\cdot\frac{r!}{k^r}\Bigl(\binom{k}{r}-\sum_{i=1}^{\infty}\binom{k/2^i}{r}\Bigr)-\frac{k}{n-k}.
\end{aligned}
\]
Note that  \(\frac{k^r}{r!} \ge \binom{k}{r}\) and \(\frac{r!}{k^r}\bigl(\frac{1}{c}\binom{ck}{r}-\frac{k}{n-k}\binom{k}{r}\bigr)\le \bigl(c^{r-1}-\frac{k}{n-k}\bigr)\frac{r!}{k^r}\binom{k}{r} \le c^{r-1}-\frac{k}{n-k}\) .

In the second case of Theorem~\ref{tight}, for sufficiently large \(n\), if \(c \ge (2r)^{1/(r-1)}\) and \(k \ge 2(r-1)\), we obtain
\[
g_{c,r}(n,k) \le c^{r-1}+\frac{1}{2c} \cdot \frac{r!}{k^r}\Bigl(\binom{k}{r}-\sum_{i=1}^{\infty}\binom{k/2^i}{r}\Bigr)-\frac{k}{n-k}.
\]
Thus to complete the proof of Theorem~\ref{main} and the upper bound in Theorem~\ref{0}, it suffices to prove
\begin{align}
\binom{k}{r}-\sum_{i=1}^{\infty}\binom{k/2^i}{r}\le \frac{k^r}{r!}\Bigl(1-\frac{1}{2^r-1}\Bigr). \tag{A.1}
\end{align}
We refer the reader to the Appendix for a complete proof of (A.1).
\end{proof}

\section{Construction for hypergraphs}\label{7}
In this section, we prove the lower bounds in Theorem~\ref{0}.
Define
\begin{equation}\label{eq:Nnkr}
N_{n,k,r} = \binom{n}{r} - \binom{n - k}{r} 
           +\frac{n-k}{ck}\binom{ck}{r} + \left( \frac{n - k}{ck} - 1 \right) \binom{k}{r} 
            - \frac{n - k}{2ck} \sum_{i = 0}^{\infty} \binom{k / 2^i}{r}    
\end{equation}
to be the upper bound of \(e(\widehat{T})\) in the second part of Theorem~\ref{tight}.


We first provide the desired lower bounds for the second part of Theorem~\ref{0}.

\begin{example}
\label{ex1}

Let \( k = 2^s r \), where \( s \ge 1 \) is an integer, and let \( n = k + \frac{2ck^2}{r} = k + 2^{s+1} ck \). We construct an \( r \)-graph \( G_{k,r} \) on \( n \) vertices with no \( (k+1) \)-connected subgraphs as follows.

Let \( H_0 \) be the complete \( r \)-graph on \( k+ck \) vertices for $c \ge 1$. For \( i \ge 1 \), define \( H_i \) recursively:
\begin{itemize}
    \item Take two disjoint copies of \( H_{i-1} \).
    \item In each copy, select a separator \( S_{i-1} \) of size \( k \) by taking exactly \( \frac{k}{2^{\,i-1}} \) vertices from each of the \( 2^{\,i-1} \) atoms of \( H_{i-1} \), with the requirement that none of these vertices has already been used in any previous separator.
    \item Glue the two \( H_{i-1} \) copies along \( S_{i-1} \).
\end{itemize}
We demonstrate that each time we glue $A$ and $B$ along their common separator $S$, bonded edges appear; that is, all edges have endpoints in each of the three parts $A \setminus S$, $B \setminus S$, and $S$. This rule also applies to the following example.
Since each atom has \( k+ck \ge 2k \) vertices, the separators can be chosen disjoint from all previously selected ones, and the number of free anti‑edges in \(S_{i}\) equals
$\binom{k}{r} - 2^{i} \cdot \binom{\frac{k}{2^{i}}}{r} , (i=0,1,\dots ,s).$
Finally, define \( G_{k,r} = H_{s+1} \).
Since each gluing step doubles the number of atoms, the hypergraph \( G_{k,r}\) has \( 2^{s+1} \) atoms and \(2^{s+1} ck + k = n\) vertices. 
For each \(0\le i\le s\), the gluing of two copies of \(H_i\) into \(H_{i+1}\) creates \(2^{s-i}\) separators at level \(i\). 
At each such separator, the \(k\) separator vertices are evenly distributed among \(2^i\) atoms in the corresponding small branch, so each class has size \(k/2^i\). 
Hence each separator at level \(i\) has \(\binom{k}{r}-2^i\binom{k/2^i}{r}\) free edges. 
Summing over all levels, we obtain
\[
\sum_{S} \bar{e}(S) = (2^{s+1} - 1)\binom{k}{r} - 2^{s} \sum_{i=0}^{s} \binom{k / 2^{i}}{r}.
\]
Moreover, since \(G_{k,r}\) has \(2^{s+1}\) atoms each of size \(ck + k\) and \(\bar{e}(A) = \bar{\beta}(S) = 0\) for every atom \(A\) and every separator \(S\), Lemma~\ref{lem:edgenum_septree} yields
\[
\begin{aligned}
e(G_{k,r}) 
&= \binom{n}{r} - \binom{n - k}{r} +2^{s+1}\binom{ck}{r}+ \left(2^{s+1} - 1\right)\binom{k}{r} 
   - 2^s \sum_{i = 0}^{s} \binom{k / 2^i}{r} \\[4pt]
&= \binom{n}{r} - \binom{n - k}{r} +\frac{n-k}{ck}\binom{ck}{r}+ \left( \frac{n - k}{ck} - 1 \right) \binom{k}{r} 
   - \frac{n - k}{2ck} \sum_{i = 0}^{\infty} \binom{k / 2^i}{r} = N_{n,k,r}.
\end{aligned}
\]
Note that \(H_{s+1}\) contains an independent set \(S\) of size \(k\), formed by selecting \(\frac{k}{2^{s+1}} = \frac{r}{2}\) vertices from each of the \(2^{s+1}\) atoms in \(H_{s+1}\). 
Then for each positive integer \(m\), we glue \(m\) copies of \(H_{s+1}\) on \(S\) and add all possible bonded edges to obtain a new hypergraph \(H_{s+1}^{(m)}\).
It is easy to verify that \(H_{s+1}^{(m)}\) has \(n_m = k + m \cdot \frac{2ck^2}{r}\) vertices and its number of edges attains the upper bound \(N_{n_m, k, r}\) in Theorem~\ref{tight}.
By Theorem~\ref{thm:limit}, \(g_{c,r}^k\) exists; therefore,
\begin{equation}\label{eq:gcrknm}
g_{c,r}^k = \lim_{m \to \infty} \left( e\bigl(H_{s+1}^{(m)}\bigr) - \binom{n_m}{r} + \binom{n_m - k}{r} \right) = \frac{r!}{k^{r-1}} \lim_{m \to \infty} \frac{r}{2m c k^2} \left( N_{n_m, k, r} - \binom{n_m}{r} + \binom{n_m - k}{r} \right).    
\end{equation}
This determines the precise value of \(g_{c,r}^{k}\) for \(k=2^s r\).
For \(2^{s-1}r < k < 2^s r\), we apply an analogous gluing procedure to construct \(H_0, H_1, \ldots, H_{s+1}\).
In the \(i\)-th layer, we choose the separator \(S_{i-1}\) by taking either \(\lfloor k/2^{i-1} \rfloor\) or \(\lceil k/2^{i-1} \rceil\) vertices from each of the \(2^{i-1}\) atoms of \(H_{i-1}\).
Consequently,
\[
\bar{e}(S_{i-1}) = \binom{k}{r} - 2^{i-1}\binom{k/2^{i-1}}{r} + O_r(k^{r-1}).
\]
Summing over all layers yields
\[
e(H_{s+1}) = N_{n,k,r} - O_r(k^{r-2}n).
\]
Since \(H_{s+1}\) contains an independent set of size \(k\), we can construct infinitely many hypergraphs whose edge number deviates from the upper bound by \(O(k^{r-2}n)\). 
This gives the \(O_r(1/k)\) estimate of \(g_{c,r}^k\) for general \(k\) and completes the proof of the second part of Theorem~\ref{0}.
\qed
\end{example}

Next, we construct examples of $r$-uniform hypergraphs that exceed the tight upper bound established for $r$-uniform hypergraphs with normal separator trees when $c=1$ (see Theorem~\ref{normal2}). 
Moreover, as $r\to\infty$, the number of edges in these constructions asymptotically attains the bound given in Theorem~\ref{0} for the case $c=1$.
These hypergraphs contain tiny atoms in their separator trees, which, in our view, indicates a significant difference between the graph case and the case $r\geq 3$.

\begin{example}\label{ex2}
When $c=1$, let \( k = 2^{s} r \) where \( s \ge 1 \) is an integer, and let
$n = k + \frac{2k^{2}}{r} + \frac{2k}{r}\cdot pk = k + 2^{s+1}(1+p)k$,
where \(pk\) is an integer and the parameter \(p\) will be specified later, with \(p \le 1\).
We construct an \(r\)-uniform hypergraph \(G_{k,r,p}'\) on \(n\) vertices that contains no \((k+1)\)-connected subgraph in 3 steps, as follows.

\begin{enumerate}
\item Let \(H_{0}\) be the complete \(r\)-uniform hypergraph on \(2k\) vertices and \(G_{0}\) the complete \(r\)-uniform hypergraph on \(k+pk\) vertices. Form \(H_{0}'\) by gluing \(H_{0}\) and \(G_{0}\) along a separator \(S\) of size \(k\). Since both \(H_{0}\) and \(G_{0}\) are complete, \(S\) is also complete.

\item Build \(H_{1}'\) by taking two disjoint copies of \(H_{0}'\) and gluing them along a separator \(S_{0}\) of size \(k\). Here \(S_{0}\) consists of \(pk\) vertices from \(G_{0}\setminus S\) and \(k-pk\) vertices from \(H_{0}\setminus S\). All anti‑edges in \(S_{0}\) are free; their number is \(\binom{k}{r} - \binom{k-pk}{r} - \binom{pk}{r}\). No edges exist between the two copies of \(S\) in \(H_{1}'\).

\item For \(i \ge 2\) define \(H_{i}'\) recursively: take two disjoint copies of \(H_{i-1}'\) and in each copy choose a separator \(S_{i-1}\) of size \(k\) that contains exactly \(\frac{k}{2^{\,i-1}}\) vertices from each of the \(2^{\,i-1}\) copies of \(S\) present in \(H_{i-1}'\). Glue the two copies along \(S_{i-1}\). Because \(S\) has \(k\) vertices, the separators \(S_{j}\;(j\in[s])\) can be chosen to be pairwise disjoint.
\end{enumerate}

Finally, set \(G_{k,r,p}' = H_{s+1}'\).  
Using Lemma~\ref{lem:edgenum_septree} again, we obtain that \(e(G'_{k,r,p})\) equals to
\begin{equation}\label{eq:e_G'_k,r,p}
\begin{aligned}
&\quad \  \binom{n}{r} - \binom{n-k}{r} + 2^{s+1}\binom{k}{r}+ 2^{s+1}\binom{pk}{r}+ \left(\Bigl(3\cdot2^s-1\Bigr)\binom{k}{r}
            - 2^s\sum_{i=0}^{\infty}\binom{k/2^{\,i}}{r}
            -2^{s}\left( \binom{k-pk}{r}+\binom{pk}{r} \right)\right)   \\
            &= \binom{n}{r} - \binom{n-k}{r} + \Bigl(\frac{5(n-k)}{2(1+p)k}-1\Bigr)\binom{k}{r} - \frac{n-k}{2(1+p)k}\binom{(1-p)k}{r}- \frac{n-k}{2(1+p)k}\sum_{i=0}^{\infty}\binom{k/2^{\,i}}{r} + \frac{n-k}{2(1+p)k}\binom{pk}{r}.
\end{aligned}
\end{equation}

Define
\[
M_{n,k,r} = \binom{n}{r} - \binom{n-k}{r} 
          + \Bigl( \frac{5(n-k)}{2k} - 1 \Bigr)\binom{k}{r}  
          - \frac{n-k}{2k} \sum_{i=0}^{\infty} \binom{k/2^{\,i}}{r}.
\]
be the upper bound of \(e(\widehat{T})\) in the first part of Theorem~\ref{tight}.
If there exists an $r$-uniform hypergraph $H$ such that
$e(H) \ge M_{n,k,r} - o_r(1) \cdot \frac{k^{r-1}}{r!} (n-k),$
then $g_{1,r}^k \ge 2 - o_{k,r}(1)$. We shall show that $G_{k,r,p}'$ asymptotically meets this bound as $r\to\infty$.
By \eqref{eq:e_G'_k,r,p}, we have
$$
e(G_{k,r,p}') - M_{n,k,r}
= -\frac{p}{1+p}\left(\frac{5(n-k)}{2k}\binom{k}{r} - \frac{n-k}{2k}\sum_{i=0}^\infty \binom{k/2^i}{r}\right)
 -\frac{1}{1+p}\frac{n-k}{2k}\binom{(1-p)k}{r}
+ \frac{1}{1+p}\frac{n-k}{2k}\binom{pk}{r}.$$
If we further choose \(p\) in the interval \(\bigl(1-(\frac{2}{r})^{1/r},\;1-(\frac{1}{r})^{1/r}\bigr)\) with \(pk\) an integer, then
\(\frac{p}{1+p}\leq p \leq 1-(\frac{1}{r})^{\frac{1}{r}} \leq 1- \exp (-\frac{\log r}{r})\leq\frac{\log r}{r}\to 0\) and \(\binom{(1-p)k}{r}\leq (1-p)^r \binom{k}{r}\leq \frac{2}{r}\cdot \frac{k^r}{r!}\). 
Hence
\[
e(G_{k,r,p}') - M_{n,k,r} \ge -O\Bigl(\frac{\log r}{r}\Bigr)\cdot \frac{k^{r-1}}{r!}\cdot(n-k).
\]
This shows that \(G_{k,r,p}'\) asymptotically attains \(M_{n,k,r}\). 
Since \(G_{k,r,p}'\) contains an independent set of size \(k\), by an argument analogous to that in Example~\ref{ex1}, for infinitely many $n$, we can construct hypergraphs whose number of edges asymptotically attains the upper bound \(M_{n,k,r}\).
Therefore, we obtain \(g^k_{1,r} \ge 2 - o_r(1)\), which completes the proof of the first part of Theorem~\ref{0}. 
\qed
\end{example}


\section{Concluding remarks}\label{sec:remarks}

In this paper, we study hypergraph analogues of Conjecture~\ref{conj:graph} and address a related question of Carmesin~\cite{3}.
We establish a general upper bound on the number of edges in an $r$-uniform hypergraph that contains no $(k+1)$-connected subgraph on more than $ck+k$ vertices, for any $c\ge 1$ and integers $k \ge r \ge 3$.
A key innovation in our proof is the estimation of the number of anti-edges destroyed after deleting all tiny atoms, which differs substantially from the approach used in the graph case in~\cite{3}.
We also present explicit constructions that asymptotically attain the upper bound when \(c=1\) and \(r\to\infty\), and that attain the upper bound when \(c>1\) and $r$ is large (more precisely, under the conditions \(c\geq (2r)^{1/(r-1)}\) and \(k \ge 2(r-1)\)).

We highlight two significant differences between the hypergraph case \(r \ge 3\) and the graph case \(r=2\). 
First, bonded edges appear, a phenomenon absent from Carmesin's original framework. 
Second, and more importantly, the structure of extremal hypergraphs is substantially different. 
As Example~\ref{ex2} illustrates, for \(c=1\) and every \(r \ge 4\), any \(r\)-uniform hypergraph with a normal separator tree is not extremal, whereas Carmesin's extremal graph \cite{3} must have a normal separator tree.

We now explain why the cut-off at \(2k\) vertices is a crucial threshold for Mader's conjecture and its hypergraph analogues. 
Both Mader's extremal graph \(G_{q,k}\) and the extremal graph in Carmesin~\cite{3} share the property that all atoms in their separator trees have size \(2k\) and are isomorphic. 
Other extremal hypergraphs with the same property can also be constructed, and these constructions extend naturally to extremal hypergraphs. 
This indicates that characterizing all extremal (hyper)graphs requires a base (hyper)graph \(H_0\) of size \(2k\) and a systematic method to glue copies of \(H_0\) together, making \(2k\) a key threshold for capturing the essential structure of the problem.

In our main results, we only consider the case \(k \ge r\). The case \(k < r\) is straightforward: a simple induction shows that
\[
e(H) \le \binom{n}{r} - \binom{n-k}{r}
\]
for any \(n\)-vertex \(r\)-uniform hypergraph \(H\) with no \((k+1)\)-connected subgraph and \(n \ge k+r-1\), since no edge can be entirely contained in a \(k\)-vertex separator. Moreover, this bound is tight, as witnessed by \(G_{q,k}\). Hence, in this case, the maximum number of edges is precisely \(\binom{n}{r} - \binom{n-k}{r}\).

In the second item of Theorem~\ref{0}, our techniques suffice in principle to obtain a tight upper bound for all $k\gg r\ge 3$, without requiring $c\ge(2r)^{1/(r-1)}$. The main difficulty lies in the complex calculation.

Finally, we conclude the paper with the following hypergraph formulation of Mader's Conjecture:
\begin{conjecture}[See Conjecture~2 in \cite{15}]\label{mader}
Let \(H\) be an \(r\)-uniform hypergraph on \(n\) vertices containing no \((k+1)\)-connected subgraph. Then for all sufficiently large \(n\),
\[
e(H) \le \binom{n}{r} - \binom{n-k}{r} 
        + \left(\frac{n}{k} - 2\right)\binom{k}{r}.
\]
\end{conjecture}
\noindent 
We can construct an \(n\)-vertex \(r\)-uniform hypergraph \(G_{q,k}^r\) achieving this bound. Let \(k,q\) be integers with \(n = kq\), and partition the vertex set into \(V_0, V_1, \dots, V_{q-1}\), each of size \(k\). The hypergraph \(G_{q,k}^r\) contains all \(r\)-subsets except those that are completely contained in \(V_0\), and except those that have no vertex in \(V_0\) but are not entirely contained in any single \(V_i\) with \(i\ge 1\). It is easy to verify that \(G_{q,k}^r\) has no \((k+1)\)-connected subgraph.
This is equivalent to asserting that, for sufficiently large $n$,
\(F_r(n,k)=\binom{n}{r}-\binom{n-k}{r}+\left(\frac{n}{k}-2\right)\binom{k}{r}.\)
Note that the leading term of $F_r(n,k)$ was determined in Theorem~\ref{thm:limit} (e.g. \eqref{equ:F_r}).

\section*{Acknowledgement}

This work is supported by National Key Research and Development Program of China 2023YFA1010201, National Natural Science Foundation of China grant 12125106, and Innovation Program for Quantum Science and Technology 2021ZD0302902. We thank Yupeng Lin for valuable discussions.

\section*{Appendix}

\subsection*{Proof of (A.1)}

\[
\binom{k}{r}-\sum_{i=1}^{\infty}\binom{k/2^i}{r}\le\frac{k^r}{r!}\Bigl(1-\frac1{2^r-1}\Bigr). \tag{A.1}
\]
\begin{proof}
Define $f(x)=\dfrac{x^r}{r!}-\dbinom{x}{r}\;(x>0)$. Then $f(x)\ge0$ and (A.1) is equivalent to
\begin{equation}\label{appendix6}
f(k)\ge\sum_{i=1}^{\infty}f(k/2^i). 
\end{equation}

We prove $f(2x)\ge2f(x)$ for all $x>0$.
We consider three cases.

\textbf{Case 1: $2x\le r-1$.} Then $\binom{2x}{r}=\binom{x}{r}=0$, so
$f(2x)-2f(x)=\dfrac{(2^r-2)x^r}{r!}\ge0$.

\textbf{Case 2: $x\ge r-1$.} Then $2x\ge r-1$. The Vandermonde identity holds as a polynomial identity in $x$; therefore,
\[
\binom{2x}{r}=\sum_{j=0}^r\binom{x}{j}\binom{x}{r-j}.
\]
Hence
\[
f(2x)-2f(x)=\frac{(2^r-2)x^r}{r!}-\sum_{j=1}^{r-1}\binom{x}{j}\binom{x}{r-j}
\ge\frac{(2^r-2)x^r}{r!}-\frac{x^r}{r!}\sum_{j=1}^{r-1}\binom{r}{j}=0.
\]

\textbf{Case 3: $(r-1)/2<x<r-1$.} Here $\binom{x}{r}=0$ and $\binom{2x}{r}\le\frac{(2x)^{r-1}(2x-(r-1))}{r!}$. The desired inequality $f(2x)\ge2f(x)$ reduces to
\[
(2^r-2)x^r\ge (2x)^{r-1}\bigl(2x-(r-1)\bigr).
\]
Set $t=x/(r-1)\in(\frac12,1)$. The inequality becomes
\[
(2^r-2)t\ge2^{r-1}(2t-1),
\]
which holds because $t/(2t-1)\ge1$ and $r\ge 3$.

From $f(2x)\ge2f(x)$ we obtain $f(x/2)\le f(x)/2$. Iterating gives $f(k/2^i)\le f(k)/2^i$ for all $i\ge1$. Summing,
\[
\sum_{i=1}^{\infty}f(k/2^i)\le f(k)\sum_{i=1}^{\infty}\frac1{2^i}=f(k),
\]
which is exactly \eqref{appendix6}. Substituting back $f(k)=k^r/r!-\binom{k}{r}$ and $f(k/2^i)=(k/2^i)^r/r!-\binom{k/2^i}{r}$ yields (A.1). 
\end{proof}

\medskip

\textit{E-mail address:} jiema@ustc.edu.cn (Jie Ma)

\medskip

\textit{E-mail address:} jeff\_776532@mail.ustc.edu.cn (Shengjie Xie)

\medskip

\textit{E-mail address:} lhotse@mail.ustc.edu.cn (Zhiheng Zheng)
\end{document}